\documentclass[11pt]{amsart}
\usepackage{amsmath, amssymb, amscd, mathrsfs, url, pinlabel,verbatim}
\usepackage[pagebackref]{hyperref}
\usepackage[margin=1.25in,marginparwidth=1in,centering,letterpaper,dvips]{geometry}
\usepackage{color,dcpic,latexsym,graphicx,epstopdf,comment}
\usepackage[all]{xy}
\usepackage[dvipsnames]{xcolor}
\usepackage{tikz,tikz-cd,pgfplots}
\usepackage{upquote,tabularx,textcomp}
\usepackage[shortlabels]{enumitem}
\usepackage[all]{hypcap}
\usepackage[color=blue!20!white,textsize=tiny]{todonotes}

\title{Small Dehn surgery and $SU(2)$}

\author[John A. Baldwin]{John A. Baldwin}
\address{Department of Mathematics \\ Boston College}
\email{john.baldwin@bc.edu}

\author{Zhenkun Li}
\address{Department of Mathematics\\Stanford University}
\email{zhenkun@stanford.edu}

\author[Steven Sivek]{Steven Sivek}
\address{Department of Mathematics\\Imperial College London}
\email{s.sivek@imperial.ac.uk}

\author{Fan Ye}
\address{Department of Pure Mathematics and Mathematical Statistics\\University of Cambridge}
\email{fy260@cam.ac.uk}

\thanks{JAB was supported by  NSF FRG Grant DMS-1952707.}

\makeatletter
\newtheorem*{rep@theorem}{\rep@title}
\newcommand{\newreptheorem}[2]{%
\newenvironment{rep#1}[1]{%
 \def\rep@title{#2 \ref{##1}}%
 \begin{rep@theorem}}%
 {\end{rep@theorem}}}
\makeatother

\newtheorem {theorem}{Theorem}
\newreptheorem{theorem}{Theorem}
\newtheorem {lemma}[theorem]{Lemma}
\newtheorem {proposition}[theorem]{Proposition}
\newtheorem {corollary}[theorem]{Corollary}
\newtheorem {conjecture}[theorem]{Conjecture}

\numberwithin{equation}{section}
\numberwithin{theorem}{section}

\theoremstyle{definition}
\newtheorem{definition}[theorem]{Definition}

\newtheorem{remark}[theorem]{Remark}
\newtheorem*{remark*}{Remark}

\newlist{pcases}{enumerate}{1}
\setlist[pcases]{
  label=\bf{Case~\arabic*:}\protect\thiscase.~,
  ref=\arabic*,
  align=left,
  labelsep=0pt,
  leftmargin=0pt,
  labelwidth=0pt,
  parsep=0pt
}
\newcommand{\case}[1][]{%
  \if\relax\detokenize{#1}\relax
    \def\thiscase{}%
  \else
    \def\thiscase{~#1}%
  \fi
  \item
}

\newcommand{\Z}{\mathbb{Z}}

\newcommand{\R}{\mathbb{R}}
\newcommand{\C}{\mathbb{C}}

\newcommand{\spc}{\operatorname{Spin}^c}

\newcommand{\Hom}{\operatorname{Hom}}

\newcommand\hf{\mathit{HF}}

\newcommand\SHI{\mathit{SHI}}

\newcommand\KHI{\mathit{KHI}}
\newcommand\HFhat{\widehat{\mathit{HF}}}
\newcommand\Is{I^\#}
\newcommand\HFK{\widehat{\mathit{HFK}}}

\DeclareFontFamily{U}{mathx}{\hyphenchar\font45}
\DeclareFontShape{U}{mathx}{m}{n}{
      <5> <6> <7> <8> <9> <10>
      <10.95> <12> <14.4> <17.28> <20.74> <24.88>
      mathx10
      }{}
\DeclareSymbolFont{mathx}{U}{mathx}{m}{n}
\DeclareFontSubstitution{U}{mathx}{m}{n}
\DeclareMathAccent{\widecheck}{0}{mathx}{"71}

\newcommand{\img}{\operatorname{Im}}

\newcommand{\ad}{\operatorname{ad}}

\newcommand{\pt}{\mathrm{pt}}

\newcommand{\cinvt}{\nu^\sharp}

\newcommand{\dcover}{\Sigma}

\newcommand{\chigr}{\chi_{\mathrm{gr}}}

\makeatletter
\DeclareFontFamily{OMX}{MnSymbolE}{}
\DeclareSymbolFont{MnLargeSymbols}{OMX}{MnSymbolE}{m}{n}
\SetSymbolFont{MnLargeSymbols}{bold}{OMX}{MnSymbolE}{b}{n}
\DeclareFontShape{OMX}{MnSymbolE}{m}{n}{
    <-6>  MnSymbolE5
   <6-7>  MnSymbolE6
   <7-8>  MnSymbolE7
   <8-9>  MnSymbolE8
   <9-10> MnSymbolE9
  <10-12> MnSymbolE10
  <12->   MnSymbolE12
}{}
\DeclareFontShape{OMX}{MnSymbolE}{b}{n}{
    <-6>  MnSymbolE-Bold5
   <6-7>  MnSymbolE-Bold6
   <7-8>  MnSymbolE-Bold7
   <8-9>  MnSymbolE-Bold8
   <9-10> MnSymbolE-Bold9
  <10-12> MnSymbolE-Bold10
  <12->   MnSymbolE-Bold12
}{}

\let\llangle\@undefined
\let\rrangle\@undefined
\DeclareMathDelimiter{\llangle}{\mathopen}%
                     {MnLargeSymbols}{'164}{MnLargeSymbols}{'164}
\DeclareMathDelimiter{\rrangle}{\mathclose}%
                     {MnLargeSymbols}{'171}{MnLargeSymbols}{'171}
\makeatother

\newcounter{desccount}

\newcommand{\descref}[1]{\hyperref[#1]{#1}}

\usetikzlibrary{calc,intersections}
\tikzset{every picture/.style=thick}
\tikzset{link/.style = { white, double = black, line width = 1.75pt, double distance = 1.25pt, looseness=1.75 }}
\tikzset{crossing/.style = {draw, circle, dotted, minimum size=0.5cm, inner sep=0, outer sep=0}}
\pgfplotsset{compat=1.12}

\begin{document}

\begin{abstract}
We prove that the fundamental group of $3$-surgery on a nontrivial knot in $S^3$ always admits an irreducible $SU(2)$-representation.
This answers a question of Kronheimer and Mrowka dating from  their work on the Property P conjecture. An important ingredient in our proof is a relationship between instanton Floer homology and the symplectic Floer homology of genus-2 surface diffeomorphisms, due to Ivan Smith. We use similar arguments  at the end to extend our main result to infinitely many   surgery slopes in the interval $[3,5)$. 
\end{abstract}

\maketitle

\section{Introduction} \label{sec:intro}
The $SU(2)$-representation variety \[R(Y) = \Hom(\pi_1(Y),SU(2))\] associated with a 3-manifold $Y$ is an object of central importance in instanton gauge theory. A basic question about such varieties is whether they contain irreducible representations. A homomorphism into $SU(2)$ is reducible if and only if it has abelian image, inspiring the following definition:

\begin{definition}
A 3-manifold $Y$ is \emph{$SU(2)$-abelian} if every $\rho\in R(Y)$ has abelian image.\footnote{If $b_1(Y)=0$ then $\rho$ has abelian image if and only if it has cyclic image.} In particular, $R(Y)$ contains an irreducible if and only if $Y$ is not $SU(2)$-abelian.
\end{definition}


The classification of $SU(2)$-abelian 3-manifolds is a wide-open problem, even among Dehn surgeries on knots.\footnote{Though see \cite{lpcz,bs-toroidal} and \cite{sz-menagerie} for  progress on this problem in the cases of  toroidal  homology spheres and non-hyperbolic geometric manifolds, respectively.} In their proof of the Property P conjecture \cite{km-p}, Kronheimer and Mrowka proved that if $K$ is a nontrivial knot in $S^3$ then $S^3_1(K)$ is not $SU(2)$-abelian (hence, not a homotopy sphere). They then strengthened this result  in \cite{km-su2}, proving under the same hypothesis that $S^3_r(K)$ is not $SU(2)$-abelian for any rational number $r\in [0,2]$. 

It is natural to ask whether this also holds for larger values of $r$. Kronheimer and Mrowka explicitly posed this question in \cite{km-su2} for the next two integers, $r=3$ and $4$, noting that it fails for $r=5$ because $5$-surgery on the right-handed trefoil is a lens space. Baldwin and Sivek answered half of this question in \cite{bs-lspace}, proving that $S^3_4(K)$ is not $SU(2)$-abelian for any nontrivial knot $K\subset S^3$. Our main theorem  answers the remaining half:

\begin{theorem}\label{thm:main}
$S^3_3(K)$ is not $SU(2)$-abelian for any nontrivial knot $K\subset S^3$.
\end{theorem}

\begin{remark}\label{rmk:other}
We ultimately expect that  $r$-surgery on a nontrivial knot is not $SU(2)$-abelian  for any  rational number $r\in [0,5)$. Our techniques can be used to prove this for infinitely many additional  slopes in the interval $[3,5)$,  as discussed in \S\ref{ssec:other}.
\end{remark}



Our proof involves a novel blend of ingredients, including  a relationship between instanton Floer homology and the symplectic Floer homology of genus-2 surface diffeomorphisms, due to   Smith  \cite{smith};  ideas of Baldwin--Hu--Sivek from their  recent work on cinquefoil detection \cite{bhs-cinquefoil}; and  new  results of Li--Ye relating the graded Euler characteristics of instanton and Heegaard    knot Floer homology \cite{li-ye-Euler, li-ye-enhanced-Euler}. 
Along the way, we partially characterize   instanton L-space knots of genus 2 (Theorem \ref{thm:genus-2-lspace} and Remark \ref{rmk:genus-2}), and  prove several other results (Proposition \ref{prop:surgery-cover}, and Theorems \ref{thm:floer-simple} and \ref{thm:classify-simple-knots}) which may be of independent interest. 
We outline our proof of Theorem \ref{thm:main} in detail below.

\subsection{Proof outline} Suppose $K\subset S^3$ is a nontrivial knot. Let us assume for a contradiction that $S^3_3(K)$ is $SU(2)$-abelian. This  surgered manifold is thus an instanton L-space \cite{bs-stein}. It then  follows from \cite{bs-lspace} that $K$ is   fibered and strongly quasipositive  of genus 1 or 2. 
If the genus is 1, then $K$ is the right-handed trefoil, but $3$-surgery on this trefoil is Seifert fibered with base orbifold $S^2(2,3,3)$ and is therefore not $SU(2)$-abelian, by \cite{sz-menagerie}, a contradiction. We similarly rule out the possibility that $K$ is the cinquefoil $T_{2,5}.$ 

Thus, $K\not\cong T_{2,5}$ is a genus-2 instanton L-space knot. 
Recent work of Li and Ye \cite{li-ye-surgery} then implies  that $K$ has Alexander polynomial \[ \Delta_K(t) = t^2-t+1-t^{-1}+t^{-2}, \] from which we conclude as in \cite{bhs-cinquefoil} that $K$ is a hyperbolic knot.  Let $(S,h)$ be an abstract open book corresponding to the fibration associated with $K$.   The hyperbolicity of $K$ implies that the monodromy $h$ is freely isotopic to a pseudo-Anosov homeomorphism $\psi:S\to S.$

In \cite{smith}, Smith proved that (a version of) the instanton Floer homology of the mapping torus of a genus-2 surface diffeomorphism encodes the symplectic Floer cohomology of the diffeomorphism. 
We  use Smith's result, together with a calculation of the framed instanton homology $\Is(S^3_0(K),\mu)$, to prove that the homeomorphism $\psi$ has no fixed points, following ideas  in \cite{bhs-cinquefoil}. We then apply results from \cite{bhs-cinquefoil}  to conclude that $K$ is doubly periodic with unknotted quotient $A$ and axis $B$, where $A$ and $B$ have linking number $5$.

We may therefore express $S^3_3(K)$ as the branched double cover 
of the lens space \[L= S^3_{3/2}(A)\cong L(3,2),\]   branched along the primitive knot $J\subset L$ induced by $B$ in this surgery on $A$, \begin{equation}\label{eqn:3-surgery-bdc}S^3_3(K) \cong \dcover(L,J).\end{equation}  Under the assumption that $S^3_3(K)$ is $SU(2)$-abelian, we  prove that every representation \[\rho:\pi_1(L\setminus N(J))\to SU(2)\] which sends a fixed meridian of $J$ to $i$ has finite cyclic image. There are exactly $|H_1(L)|=3$ such representations.
We prove that these representations are nondegenerate as generators of the instanton knot homology of $J$, and thereby conclude that
$J$ is an \emph{instanton Floer simple knot}, meaning that \[\dim_\C\KHI(L,J) = \dim_\C \Is(L) = 3.\] We deduce   that $J$ is  rationally fibered, and that the graded Euler characteristic of $\KHI(L,J)$ is given by  \[\chigr(\KHI(L,J)) = t^{-n} + 1 + t^n,\] for some $n = g(F)+1,$ where $F$ is a minimal-genus rational Seifert surface for $J$.

In recent work \cite{li-ye-Euler}, Li and Ye proved an equality between the  graded Euler characteristics of instanton knot homology and (Heegaard) knot Floer homology, \[\chigr(\HFK(L,J))= \pm \chigr(\KHI(L,J)).\] We combine this with constraints on the graded Euler characteristic of knot Floer homology to conclude that $n=1$. Thus, $J$ has (rational) genus 0, which implies that its fibered exterior is  a solid torus.  It follows that $\dcover(L,J)$, and hence $3$-surgery on $K$, is a lens space.  Since $K$ is doubly periodic and has a lens space surgery, a result of Wang and Zhou \cite{wang-zhou} says that $K$ must be a torus knot, a contradiction.

\subsection{Other small surgeries}\label{ssec:other}
As mentioned above, we suspect that if $K\subset S^3$ is a nontrivial knot, then $S^3_r(K)$ is not $SU(2)$-abelian for any rational  $r\in [0,5)$. Long after Kronheimer and Mrowka's proof for $r\in [0,2]$, Baldwin and Sivek proved this  \cite{bs-lspace} for the dense subset of  the interval $(2,3)$ consisting of rationals with prime power numerators. Proving it for the corresponding subset of $[3,5)$ is much harder, as explained in \S\ref{sec:other}. Indeed, our main theorem concerns the single  case $r=3$.
  Nevertheless, our techniques can  be used to extend Theorem \ref{thm:main} to infinitely many surgery slopes in both  $[3,4)$ and $[4,5)$, and even to a special dense subset of slopes in $(5,7)$.  Our main result in this vein is:

\begin{theorem}
\label{thm:other-small-surgeries}  Let $p$ and $q$ be coprime positive integers with $p/q\in [3,7)$, where  $p$ is a prime power. If either
\begin{itemize}
\item $p$ is even,
\item $p$ is odd and $\gcd(p,5)=1$ and $p/q\in [4,5)$, or
\item $p$ is odd and $\gcd(p,5)=1$ and $p/q\in [3,4)$ with $23p-9 \leq 80q \leq 25p+9$,
\end{itemize}
then $S^3_{p/q}(K)$ is not $SU(2)$-abelian for every nontrivial knot $K\subset S^3$.
\end{theorem}

The inequality in Theorem \ref{thm:other-small-surgeries} is satisfied for $p/q\in [16/5, 80/23)\subset [3,4)$. It is  satisfied for $p/q = 3$ as well, 
so this theorem recovers Theorem~\ref{thm:main}. Note that Theorem \ref{thm:other-small-surgeries}  also applies to the slopes $4,$ $7/2$, and $9/2$. Together with the results mentioned above, concerning the intervals $[0,2]$ and $(2,3)$, this proves that:

\begin{corollary}
If $K\subset S^3$ is a nontrivial knot then $S^3_r(K)$ is not $SU(2)$-abelian for any integer or half-integer $r<5$.
\end{corollary}




\subsection{Organization} In \S\ref{sec:hyperbolic}, we prove that if $K\not\cong T_{2,5}$ is an instanton L-space knot of genus 2, then $K$ is doubly-periodic with unknotted quotient. We then use this in \S\ref{sec:3-surgery} to show that if $3$-surgery on a nontrivial knot  is $SU(2)$-abelian, then this surgery is the branched double cover of a   knot $J\subset L(3,2)$; in fact, we prove something  more general (Proposition \ref{prop:surgery-cover}). In \S\ref{sec:floer-simple}, we prove  that such $J$ are instanton Floer simple (Theorem \ref{thm:floer-simple}). In \S\ref{sec:l32}, we  combine these results  with an understanding of  instanton Floer simple knots in lens spaces (Theorem \ref{thm:classify-simple-knots}) to prove Theorem \ref{thm:main}. Finally, in \S\ref{sec:other}, we  extend  our main theorem to other surgery slopes in $[3,5)$, proving Theorem \ref{thm:other-small-surgeries}.

\subsection{Acknowledgements} We thank Ying Hu for helping to pioneer   the ideas in \cite{bhs-cinquefoil} which inspired many of the proofs in this paper. We also thank Ken Baker, Yi Ni, and Jake Rasmussen for helpful discussions. Finally, thanks to Peter Kronheimer and Tom Mrowka for posing the interesting question at the heart of this work, and for their encouragement.

\section{Genus two L-space knots} \label{sec:hyperbolic}

Recall that a rational homology 3-sphere $Y$ is  an \emph{instanton L-space} if its framed instanton homology satisfies
\[ \dim_\C I^\#(Y) = |H_1(Y)|. \]
A knot $K\subset S^3$ is   said to be an \emph{instanton L-space knot} if $S^3_r(K)$ is an instanton L-space for some rational number $r>0$. Such knots were studied extensively in \cite{bs-lspace}; in particular, it was shown there \cite[Theorem 1.15]{bs-lspace} that  instanton L-space knots are fibered and strongly quasipositive. 

Our main goal in this section is to prove the following: 

\begin{theorem} \label{thm:genus-2-lspace}
Let $K \subset S^3$ be an instanton L-space knot of  genus 2.  Then $K$ is fibered and strongly quasipositive, with Alexander polynomial
\[ \Delta_K(t) = t^2-t+1-t^{-1}+t^{-2}. \]
If $K \not\cong T_{2,5}$, then there exists a pseudo-Anosov 5-braid $\beta$ whose closure $B=\hat\beta$ is an unknot with braid axis $A$, such that $K$ is the lift of $A$ in the branched double cover
\[ \dcover(S^3,B) \cong S^3. \]
In particular, $K$ is a doubly periodic knot with unknotted quotient $A$ and axis $B$.
\end{theorem}

A Heegaard Floer analogue of this result was proved in \cite[\S3]{bhs-cinquefoil}. Our proof of Theorem \ref{thm:genus-2-lspace}  will largely follow the arguments there, and we  will refer to specific parts of \cite{bhs-cinquefoil} for steps whose proofs are the same in either setting.
We begin with some preliminary results.

\begin{lemma} \label{lem:i-sharp-basics}
Let $K$ be an instanton L-space knot of  genus 2.  Then $K$ is fibered and strongly quasipositive, with Alexander polynomial
\[ \Delta_K(t) = t^2 - t + 1 - t^{-1} + t^{-2}, \]
and either $K \cong T_{2,5}$ or $K$ is hyperbolic.
\end{lemma}

\begin{proof}
The claims that $K$ is fibered and strongly quasipositive are \cite[Theorem~1.15]{bs-lspace}, as discussed above.  The Alexander polynomial $\Delta_K(t)$ is determined by \cite[Theorem~1.9]{li-ye-surgery}, which says that there are integers $k \geq 0$ and
\[ n_k > n_{k-1} > \dots > n_0=0 > n_{-1} > \dots > n_{-k}, \]
with $n_{-j} = -n_j$ for all $j$, such that the instanton knot homology $\KHI(K)$ satisfies
\[ \dim_\C \KHI(K,i) = \begin{cases} 1, & i = n_j \text{ for some j} \\ 0, & \text{otherwise}, \end{cases} \]
with the nonzero summands $\KHI(K,n_j)$ supported in alternating $\Z/2\Z$ gradings.  We know that $\KHI(K,2) \neq 0$ since $K$ has genus 2 \cite[Proposition~7.16]{km-excision}, and then $\KHI(K,1) \neq 0$ since $K$ is fibered \cite[Theorem~1.7]{bs-trefoil}, so $k=2$ and $(n_2,n_1,n_0)=(2,1,0)$.  Then $\KHI(K)$ has graded Euler characteristic
\[ \sum_{j\in\Z} \chi(\KHI(K,j)) \cdot t^j = \pm \Delta_K(t) \]
\cite{km-alexander,lim}, so the normalization $\Delta_K(1)=1$ forces $\Delta_K(t)$ to be as claimed.  Since $K$ is fibered with this particular Alexander polynomial, it is either $T_{\pm2,5}$ or hyperbolic by \cite[Lemma~3.2]{bhs-cinquefoil}. The case  of $T_{-2,5}$ is ruled out as this knot is not strongly quasipositive.\end{proof}

\begin{lemma} \label{lem:i-zero-surgery}
Let $K$ be an instanton L-space knot of genus 2, and let $w \to S^3_0(K)$ be a Hermitian line bundle such that $c_1(w)$ is Poincar\'e dual to the meridian $\mu$ of $K$.  Then
\[ \dim_\C I_*(S^3_0(K))_w = 6, \] 
where $I_*$ refers here to the version of instanton  homology for admissible bundles described, for example, in \cite[\S5.6]{donaldson-book}.
\end{lemma}

\begin{proof}
We know that $3$-surgery on $K$ is an instanton L-space \cite[Proposition~7.11]{bs-lspace}, while neither $1$- nor $2$-surgery on $K$ is an instanton L-space \cite[Proposition~7.12]{bs-lspace}.  It follows in the language of \cite{bs-concordance} that $\cinvt(K) = r_0(K) = 3$, meaning that
\[ \dim_\C I^\#(S^3_{p/q}(K)) = \begin{cases} p, & p/q \geq 3 \\ 6q-p, & p/q < 3 \end{cases} \]
for all coprime integers $p,q$ with $q \geq 1$.  In particular we have
\[ \dim_\C I^\#(S^3_0(K),\mu) = 6 \]
by \cite[Proposition~3.3]{bs-concordance}.

To prove the lemma, we then  relate   $I_*(S^3_0(K))_w$ to $I^\#(S^3_0(K),\mu)$ using Fukaya's connected sum theorem for instanton homology \cite[Theorem~1.3]{scaduto}.  To be precise, since $H_2(S^3_0(K))$ is generated by a closed surface of genus 2 which is dual to $\mu$, the special case of this theorem discussed in \cite[\S9.8]{scaduto} says that
\[ I^\#(S^3_0(K),\mu) \otimes H_*(S^4) \cong I_*(S^3_0(K))_w \otimes H_*(S^3) \]
as relatively $\Z/4\Z$-graded vector spaces over any field of characteristic zero.  It follows that $\dim_\C I_*(S^3_0(K))_w = 6$ as claimed.
\end{proof}

Recall for the proposition below that instanton L-space knots are fibered.

\begin{proposition} \label{prop:l-space-monodromy}
Let $K\not\cong T_{2,5}$ be an instanton L-space knot of genus 2. Let   $(S,h)$ be an    open book encoding the fibration of its complement. Then the monodromy $h$ is freely isotopic to a pseudo-Anosov homeomorphism \[ \psi: S \to S \] with no fixed points.
\end{proposition}

\begin{proof}

The complement  $S^3\setminus K$ is hyperbolic, by Lemma  \ref{lem:i-sharp-basics}. It follows that the monodromy $h$ is  freely isotopic to  a pseudo-Anosov homeomorphism
\[ \psi: S \to S, \]
by work of Thurston \cite{thurston-fiber}.  Since $K$ is strongly quasipositive, also by Lemma  \ref{lem:i-sharp-basics},  \cite[Lemma~3.3]{bhs-cinquefoil} says that   the stable  foliation of $\psi$ has some number $n \geq 2$ of boundary prongs. 

As discussed in \cite[\S2]{bhs-cinquefoil}, the fact that $n\geq 2$ means that the map $\psi$  extends naturally to a pseudo-Anosov homeomorphism \[\hat\psi:\hat S\to \hat S\] of the closed genus-2 surface obtained from $S$ by capping off its boundary with a disk. The stable foliation of $\psi$ extends across the disk to a stable invariant foliation for $\hat\psi$  in which the $n$ boundary prongs extend to $n$ prongs meeting at a singularity $p$ in the capping disk (except that $p$ is a smooth point if $n=2$). Moreover, $p$ is a fixed point of $\hat\psi$.

Note that the mapping torus $Y_{\hat\psi}$ of $\hat\psi$ is homeomorphic to $0$-surgery on $K$, \[Y_{\hat\psi}\cong S^3_0(K),\] by a map which identifies the suspension of $p$ with the   meridian $\mu$ of $K$ in the $0$-surgery. Smith \cite[Corollary~1.8]{smith} proved in this case (or, more generally, for any genus-2 mapping torus which is a homology $S^1\times S^2$ with a distinguished section coming from a marked point on the surface) that there is an isomorphism (with $\C$ coefficients) of the form
\begin{equation}
\label{eqn:isofixed} \hf_{\mathrm{inst}}(E \to Y_{\hat{\psi}}) \cong \C \oplus \hf_{\mathrm{symp}}^\ast(\hat{\psi}) \oplus \C. \end{equation}
Here, $E \to Y_{\hat{\psi}}$ is the nontrivial $SO(3)$ bundle with $w_2$ dual to $\mu$;    $\hf^\ast_{\mathrm{symp}}(\hat{\psi})$ is the symplectic Floer cohomology of $\hat{\psi}$, defined via any area-preserving diffeomorphism (for any area form on $\hat S$) in the mapping class of $\hat\psi$; and $\hf_{\mathrm{inst}}$ is the version of instanton homology used by Dostoglou and Salamon in \cite{dostoglou-salamon}. The latter is defined using a slightly larger gauge group than the determinant-1 gauge transformations defining $I_*$, and is thus the quotient of the relatively $\Z/8\Z$-graded $I_*$ by a degree-4 involution. It follows, in particular, that \[\dim_\C \hf_{\mathrm{inst}}(E \to Y_{\hat{\psi}}) = \frac{1}{2}\cdot \dim_\C I_*(S^3_0(K))_w,\] where $w \to S^3_0(K)$ is a Hermitian line bundle with $c_1$   dual to  $\mu$, as in Lemma \ref{lem:i-zero-surgery}.
The same lemma says that $\dim_\C I_*(S^3_0(K))_w = 6$, from which it follows that
\[ \dim_\C \hf_{\mathrm{inst}}(E\to Y_{\hat{\psi}}) = 3, \]
and therefore that \begin{equation}
\label{eqn:symp}\dim_\C \hf_{\mathrm{symp}}^\ast(\hat{\psi}) = 1,\end{equation} by the isomorphism \eqref{eqn:isofixed}.

The symplectic Floer cohomology of pseudo-Anosov homeomorphisms was computed by Cotton-Clay in \cite[\S3]{cotton-clay}. In particular, given  $\hat\psi$ as above, Cotton-Clay defines in \cite[\S3.2]{cotton-clay} a canonical smooth representative $\hat\psi_{{sm}}$ of $\hat\psi$ whose symplectic Floer chain complex \[\mathit{CF}_\ast^{\mathrm{symp}}(\hat\psi_{{sm}})\] is freely generated by the fixed points of $\hat\psi_{{sm}}$ and has trivial differential. This implies that  \begin{equation}\label{eqn:symprep} \hf^\ast_{\mathrm{symp}}(\hat{\psi}) \cong \mathit{CF}_\ast^{\mathrm{symp}}(\hat\psi_{{sm}}). \end{equation}
(This is proved in \cite{cotton-clay} with $\Z/2\Z$ coefficients, but also holds in characteristic zero since the differential vanishes for purely topological reasons.) Furthermore, $\hat\psi_{{sm}}$ has the property that each fixed point of $\hat\psi$ is also a fixed point of $\hat\psi_{{sm}}$. The combination of \eqref{eqn:symp} and \eqref{eqn:symprep} therefore implies that $\hat\psi$ has at most one fixed point. On the other hand, we know that $p$ is a fixed point of $\hat\psi$ from the discussion above. It follows that this is the only fixed point of $\hat\psi$, which implies that the original map $\psi$ has no fixed points.
\end{proof}

\begin{remark}
The proof of Proposition~\ref{prop:l-space-monodromy} is notably easier than that of the analogous \cite[Theorem~3.5]{bhs-cinquefoil}, at least modulo deep theorems by others.  The reason is that the work of Lee and Taubes \cite{lee-taubes} used in \cite{bhs-cinquefoil} requires a monotonicity condition which only applies in genus 3 or higher, so a trick is required to pass from the genus 2 case to the higher genus case.  By contrast, the references \cite{dostoglou-salamon,smith} work perfectly well for genus-2 mapping tori, and in fact \cite[Corollary~1.8]{smith} \emph{only} applies in genus 2.
\end{remark}

\begin{remark}\label{rmk:genus-2}
Suppose $K\not\cong T_{2,5}$ is a genus-2 instanton L-space knot,  $(S,h)$ is an open book associated to the fibration of its complement, and $\psi: S \to S$ a pseudo-Anosov map isotopic to $h$.  Proposition~\ref{prop:l-space-monodromy} says that $\psi$ has no fixed points.  Thus we can apply \cite[Proposition~3.8]{bhs-cinquefoil}, whose other hypotheses follow from Lemma~\ref{lem:i-sharp-basics}, to conclude the following:
\begin{itemize}
\item the monodromy $h$ has fractional Dehn twist coefficient $c(h)=\frac{1}{4}$;
\item the stable foliation of $\psi$ has four prongs on $\partial S$, and two interior 3-pronged singularities which are exchanged by $\psi$.
\end{itemize}
We will not need either of these properties in this article, but this may be independently useful towards a characterization of genus-2 instanton L-space knots.
\end{remark}

At this point, the proof of Theorem~\ref{thm:genus-2-lspace} proceeds in exactly the same way as that of its Heegaard Floer analogue \cite[Theorem 3.1]{bhs-cinquefoil}.

\begin{proof}[Proof of Theorem~\ref{thm:genus-2-lspace}]
The first few claims follow immediately from Lemma~\ref{lem:i-sharp-basics}.  To argue that if $K \not\cong T_{2,5}$ then there is  a pseudo-Anosov 5-braid $\beta$ with unknotted closure $B=\hat{\beta}$ such that $K$ is the lift of the braid axis $A$ to the branched double cover of $B$, we simply repeat the proof of \cite[Theorem~3.1]{bhs-cinquefoil} verbatim.  That proof relies only on the conclusions of Proposition~\ref{prop:l-space-monodromy}, so it applies equally well here.
\end{proof}

\section{$SU(2)$-abelian surgeries and branched double covers} \label{sec:3-surgery}

The conditions of being an instanton L-space and being $SU(2)$-abelian are closely related. For surgeries on knots, we have the following:

\begin{lemma} \label{lem:su2-cyclic-lspace}
Let $K \subset S^3$ be a knot, and suppose $p$ and $q$ are positive coprime integers such that $p$ is a prime power. If \[Y=S^3_{p/q}(K)\] is $SU(2)$-abelian, then $Y$  is an instanton L-space.
\end{lemma}

\begin{proof}
An $SU(2)$-abelian rational homology sphere $Y$ is an instanton L-space whenever $\pi_1(Y)$ is \emph{cyclically finite}, by \cite[Theorem~4.6]{bs-stein}.
This cyclical finiteness is satisfied if, for example, $H_1(Y)$ is cyclic of prime power order, by \cite[Proposition~4.9]{bs-stein}.
\end{proof}

Our main goal is to prove that 3-surgery on a nontrivial knot in $S^3$ is never $SU(2)$-abelian.  In this section we use Lemma \ref{lem:su2-cyclic-lspace}, together what we proved about genus-2 instanton L-space knots in \S\ref{sec:hyperbolic}, to reduce the problem to one about branched double covers of  knots in the lens space $L(3,2)$, as outlined in the introduction. In fact, we prove a substantially more general result in Proposition \ref{prop:surgery-cover} below. We will use this additional generality in \S\ref{sec:other} to address other surgeries.

\begin{proposition}
\label{prop:surgery-cover}
Let $K\subset S^3$ be a nontrivial knot, and suppose $p,q$ are coprime positive integers with $p/q<5$, where $p$ is an odd prime power.
If \[Y=S^3_{p/q}(K)\] is $SU(2)$-abelian, then: \begin{itemize}
\item $Y$ is not Seifert fibered,
\item $K$ is hyperbolic, and doubly periodic with unknotted quotient $A$, and
\item $Y$ is homeomorphic to a branched double cover of the lens space \[L= S^3_{p/2q}(A)= L(p,2q)\] along some knot $J\subset L$  which is homologous to 5 times the core of the solid torus $S^3\setminus N(A)\subset L$ in the genus-1 Heegaard splitting of $L$ along the  surface $\partial N(A)$.
\end{itemize}
\end{proposition}

\begin{proof}
Lemma \ref{lem:su2-cyclic-lspace} implies that $Y$ is an instanton L-space; in particular,  $K$ is an instanton L-space knot. It then follows from \cite[Theorem~1.15]{bs-lspace} that \[2g(K)-1\leq p/q<5,\] which implies that $K$ has genus 1 or 2.

We next observe that $K$ is neither $T_{2,3}$ nor $T_{2,5}$. Indeed, $p/q<5$ implies that \[|abq-p | > 1, \textrm{ for }(a,b)=(2,3)\textrm{ and }(2,5).\]  Then $(p/q)$-surgeries on $T_{2,3}$ and $T_{2,5}$ are Seifert manifolds which are  not lens spaces \cite{moser}. Such surgeries cannot be $SU(2)$-abelian  \cite[Remark 1.3]{sz-menagerie}. We conclude that $K$ is neither $T_{2,3}$ nor $T_{2,5}$. Since $T_{2,3}$ is the only genus-1 instanton L-space knot \cite[Proposition~7.12]{bs-lspace}, it follows that $K$ must be a genus-2 instanton L-space knot other than $T_{2,5}$.

Theorem~\ref{thm:genus-2-lspace} then says that $K$ is hyperbolic, and doubly periodic with unknotted quotient $A$ and axis $B$, where $A$ and $B$ have linking number $5$.  In particular, there is an involution $\tau:S^3\to S^3$ with $\tau(K) = K$ whose fixed set is the axis $B$. This involution extends from the exterior $S^3 \setminus N(K)$ across the solid torus realizing the Dehn surgery producing \[Y=S^3_{p/q}(K),\] and it acts freely on this solid torus since $p$ is odd.  Thus, $\tau$ induces  an involution on $Y$ with fixed set $B$; the quotient of $Y$ by this action is then the lens space \[L=S^3_{p/2q}(A)\cong L(p,2q)\] obtained as $(p/2q)$-surgery on the unknotted quotient  $A$. In particular,  $Y$ is homeomorphic to a branched double cover of $L$ along the image $J$ of $B$ in this lens space.
The fact that $A$ and $B$ have linking number 5 implies that $J$ is homologous in $H_1(L)$ to 5 times the core of the solid torus $S^3\setminus N(A)$ in the genus-1 Heegaard splitting of $L$ along the  surface $\partial N(A)$.

Finally, we argue that \[Y=S^3_{p/q}(K)\] is not Seifert fibered. Suppose that it is. Then $Y$ is a lens space, by \cite[Remark 1.3]{sz-menagerie}. In this case, since $K$ is doubly periodic (and thus admits a symmetry with 1-dimensional fixed point set which is not a strong inversion) and has a nontrivial cyclic surgery, a theorem of Wang and Zhou \cite[Proposition 3]{wang-zhou} says that $K$ must be a torus knot, a contradiction.
\end{proof}




%
%
%
%
%

\section{Branched double covers and  Floer simple knots} \label{sec:floer-simple}

Our  goal in this section is to prove Theorem \ref{thm:floer-simple} below. This result, in combination with Proposition \ref{prop:surgery-cover}, implies that if $3$-surgery on a nontrivial knot is $SU(2)$-abelian, then the surgered manifold is the branched double cover of an instanton Floer simple knot in  $L(3,2)$. We will combine this with an understanding of Floer simple knots in this lens space in \S\ref{sec:l32} to complete the proof of Theorem \ref{thm:main}, as outlined in the introduction, and we will apply Theorem \ref{thm:floer-simple} to infinitely many additional surgery slopes in \S\ref{sec:other}. The various results in this section may be of independent interest.

\begin{theorem} \label{thm:floer-simple}
Suppose $L$ is a rational homology sphere in which $H_1(L)$ is cyclic of order an odd prime power, and let $J \subset L$ be a primitive knot.  If the branched double cover $\dcover(L,J)$ is $SU(2)$-abelian and satisfies
\[ |H_1(\dcover(L,J))|=|H_1(L)|, \]
then
\[ \dim_\C \KHI(L,J) = |H_1(L)|. \]
In other words, the knot $J \subset L$ is instanton Floer simple.
\end{theorem}

We start by clarifying some of the terminology and notation in the statement  of Theorem \ref{thm:floer-simple}, explaining in particular why a primitive knot has a unique branched double cover.
 
For the rest of this section, $L$ will refer to a rational homology sphere, and $J$ to a knot in $L$. We say that $J\subset L$   is  \emph{primitive} if its homology class $[J]$ generates $H_1(L;\Z)$. In this case, an exercise in algebraic topology shows that \[H_1(L\setminus N(J)) \cong \Z.\] This group is generated by the  homology class of some peripheral curve $\alpha \subset \partial N(J)$.  Dehn filling along $\alpha$ produces an integral homology sphere $Z$, and if $C$ is the core of this filling, then $\alpha$ is a meridian $\mu_C$ of $C$, and we have a homeomorphism
\[ Z \setminus N(C) \cong L \setminus N(J). \] Since Dehn filling $Z \setminus N(C)$ along the meridian $\mu_J$ of $J$ recovers $L$, it follows that \begin{equation*}\label{eqn:muj}[\mu_J] = |H_1(L)|\cdot [\mu_C] \end{equation*} in the first homology of this knot complement. We will use the notation established in this paragraph throughout the rest of this section.

\begin{lemma} \label{lem:h1-bdc}
Suppose $H_1(L)$ has odd order and that $J \subset L$ is primitive, and let $C \subset Z$ be as above.  Then  $J \subset L$ and $C \subset Z$ have unique branched double covers,   $\dcover(L,J)$ and $\dcover (Z,C)$, which satisfy 
\[ H_1(\dcover(L,J)) \cong H_1(\dcover(Z,C)) \oplus H_1(L). \]
\end{lemma}

\begin{proof}
Since the knot $C \subset Z$ is nullhomologous,  we can take a meridian $\mu_C$ and longitude $\lambda_C$ in $\partial N(C)$ and write
\[ \mu_J = (\mu_C)^p (\lambda_C)^q \]
in $\pi_1(\partial N(C))$ for some relatively prime integers $p$ and $q$, where we reverse the orientation of $C$ if necessary to arrange $p \geq 0$.  Then $H_1(L) \cong \Z/p\Z$.

To define $\dcover(L,J)$, we first consider the unique connected double cover $M \to L\setminus N(J)$ of $L\setminus N(J)$, specified by the kernel of the map
\[ \pi_1(L \setminus N(J)) \to H_1(L \setminus N(J)) \cong \Z \xrightarrow{\text{mod }2} \Z/2\Z. \]
Since $[\mu_C]$ generates $ H_1(L \setminus N(J)) $ and  \[[\mu_J] = |H_1(L)|\cdot [\mu_C],\]  with  $|H_1(L)|$ odd, it follows that $\mu_J$ is sent to $1$ under this map. Thus, $\mu_J^2$ lifts to a simple closed curve in $M$. The unique branched double cover $\Sigma(L,J)$ is then formed by Dehn filling $M$ along this lift of $\mu_J^2$. We similarly define $\dcover(Z,C)$ by Dehn filling $M$ along a lift of $\mu_C^2$.

Since $C$ is nullhomologous in $Z$, its lift $\tilde{C} \subset \dcover(Z,C)$ is also nullhomologous, so we define $\mu_{\tilde{C}}$ and $\lambda_{\tilde{C}}$ to be its meridian and longitude in the boundary of \[M = \dcover(Z,C) \setminus N(\tilde{C}).\] These curves lift $\mu_C^2$ and $\lambda_C$.  If $\mu_J = (\mu_C)^p(\lambda_C)^q$ as above, then $\mu_J^2 = (\mu_C^2)^p(\lambda_C)^{2q}$ lifts to
\[ \mu_{\tilde{J}} = (\mu_{\tilde{C}})^p (\lambda_{\tilde{C}})^{2q} \]
and so $\dcover(L,J)$ is the result of $(p/2q)$-surgery on $\tilde{C} \subset \dcover(Z,C)$.  Since $\tilde{C}$ is nullhomologous and $H_1(L) \cong \Z/p\Z$, this completes the proof.
\end{proof}

\begin{remark} We will often use the shorthand $\Sigma(J)$ for $\Sigma(L,J)$ out of convenience.
\end{remark}

The lemma below will be useful in several places.

\begin{lemma}\label{lem:h1-bdc-alex}
Suppose $H_1(L)$ has odd order and  $J \subset L$ is primitive. Then 
\[ |H_1(\dcover(L,J))|  = |H_1(L)| \cdot |\Delta_J(-1)|, \] where $\Delta_J(t)$ is the symmetrized Alexander polynomial of $J$.
\end{lemma}

\begin{proof} Let $C\subset Z$ be as above.
By \cite[Theorem~8.21]{burde-zieschang}, we have that \[ |H_1(\dcover(Z,C))| = |\Delta_C(-1)| = |\Delta_J(-1)|,\]
where the second equality follows from the fact that $C$ and $J$ have homeomorphic complements and the Alexander polynomial depends only on the complement. Then  \[|H_1(\dcover(L,J))| = |H_1(L)|\cdot |H_1(\dcover(Z,C))| = |H_1(L)|\cdot |\Delta_J(-1)|,\] by  Lemma \ref{lem:h1-bdc}.
\end{proof}

In the results below, we view $SU(2)$ as the group of unit quaternions.

\begin{proposition} \label{prop:h1-bdc-equals-h1-l}
Suppose $H_1(L)$ has odd order, that $J \subset L$ is primitive, and that \[|H_1(\dcover(L,J))|=|H_1(L)|.\]  Then every representation
\[ \rho: \pi_1(L \setminus N(J)) \to SU(2) \]
with image in the binary dihedral group $\{e^{i\theta}\} \cup \{e^{i\theta}j\}$ actually has cyclic image.  Moreover, there are exactly $|H_1(L)|$ such representations satisfying $\rho(\mu_J) = i$.
\end{proposition}

\begin{proof}
Let $C\subset Z$ be as above, so that $[\mu_C]$ generates $H_1(L\setminus N(J)) \cong \Z$ and \[[\mu_J] = |H_1(L)|\cdot [\mu_C].\]
By Lemma~\ref{lem:h1-bdc-alex}, we have $|\Delta_J(-1)| = |\Delta_C(-1)|=1$.  Therefore $|H_1(\dcover(C))|=1$.
 It follows that there are no non-abelian representations
\[ \pi_1(Z \setminus N(C)) \to SU(2) \]
with binary dihedral image, since the number of such conjugacy classes is
\[ \frac{|H_1(\dcover(C))|-1}{2}. \]
This was originally proved for knots in $S^3$ by Klassen \cite[Theorem~10]{klassen}, whose proof makes essential use of the Wirtinger presentation; Boden and Friedl \cite[Corollary~1.3]{boden-friedl-1} give a generalization whose proof applies equally well to knots in any integral homology sphere.  In any case, since $L\setminus N(J) \cong Z\setminus N(C)$, we conclude that there are no non-abelian binary dihedral representations of $\pi_1(L \setminus N(J))$ either. In particular, every binary dihedral representation of $\pi_1(L \setminus N(J))$ has image in the cyclic group $\{e^{i\theta}\}$.

If $\rho$ is such a representation, then (since it has abelian image) it factors through a map \[H_1(L\setminus N(J)) \cong \Z \to SU(2),\] whose domain is generated by $[\mu_C]$.  It follows that $\rho$ is uniquely determined by its evaluation at $\mu_C$, and that
\[ \rho(\mu_J) = \rho(\mu_C)^{|H_1(L)|}. \]
If we require that $\rho(\mu_J) = i = e^{i \pi/2}$, then the possible $\rho$ are specified by
\[ \rho(\mu_C) = \exp\left(i\cdot \frac{\frac{\pi}{2}+2\pi m}{|H_1(L)|}\right), \quad m=0,1,\dots,|H_1(L)|-1, \] and the proposition follows.
\end{proof}

The following is a mild adaptation of \cite[Proposition~3.1]{zentner-simple}.

\begin{lemma} \label{lem:image-ad-rho}
Suppose  $H_1(L)$ has odd order, that $J \subset L$ is primitive, and that $\dcover(L,J)$ is an $SU(2)$-abelian rational homology sphere.  Let
\[ \rho: \pi_1(L \setminus N(J)) \to SU(2) \]
be a representation satisfying $\rho(\mu_J)=i$.  Then the image of
\[ \ad\rho: \pi_1(L\setminus N(J)) \to SO(3) \]
is either cyclic or dihedral, and has order $2n$ for some odd integer  $n$ dividing $|H_1(\dcover(J))|$.
\end{lemma}

\begin{proof}
We have a short exact sequence of groups
\[ 1 \to \pi_1(M) \xrightarrow{p_*} \pi_1(L\setminus N(J)) \xrightarrow{m} \Z/2\Z \to 1 \]
where \[M\xrightarrow{p}L\setminus N(J)\] is the unique connected double covering.  The class $\mu_J^2$ belongs to $\ker(m) = p_*\pi_1(M)$, since the homology class $[\mu_J^2] = 2[\mu_J]$ is zero mod 2, so we can view the normal closure $\llangle \mu_J^2\rrangle$ as a subgroup of $\pi_1(M)$.  Upon passing to quotients, we obtain a short exact sequence
\[ 1 \to \pi_1(\dcover(J)) \to \frac{\pi_1(L\setminus N(J))}{\llangle \mu_J^2\rrangle} \to \Z/2\Z \to 1. \]

Since  $\dcover(J)$ is $SU(2)$-abelian, and $H_1(\dcover(J))$ has odd order by Lemma~\ref{lem:h1-bdc} and the oddness of $|H_1(L)|$, we may conclude  \cite[Lemma~3.2]{zentner-simple}  that every representation \[\pi_1(\dcover(J)) \to SO(3)\] has cyclic image.  (To see this,  note that every representation $\pi_1(\dcover(J)) \to SO(3)$ lifts to an $SU(2)$ representation, since the obstruction to lifting belongs to
\[ H^2(\dcover(J);\Z/2\Z) = 0. \]
The image of this lift is abelian by assumption---and, hence, cyclic, since $b_1(\dcover(J)) = 0$---so the projection of this image back down to $SO(3)$ must then be cyclic as well.)  Any such representation therefore factors through $H_1(\dcover(J))$, which is finite of odd order, and so its image also has odd order.

Now given \[\rho: \pi_1(L \setminus N(J)) \to SU(2)\] with $\rho(\mu_J) = i$, we observe that $\ad\rho$ sends $\mu_J^2$ to $\ad(i^2) = \ad(-1) = 1$, hence descends to a representation
\[ \tilde\rho: \pi_1(L\setminus N(J))/\llangle \mu_J^2\rrangle \to SO(3) \]
with $\img(\tilde\rho) = \img(\ad\rho)$.  The  short exact sequence above yields a commutative diagram
\[ \begin{tikzcd}
1 \ar[r] & \pi_1(\dcover(J)) \ar[dr,dashed,"r"'] \ar[r] & \frac{\pi_1(L\setminus N(J))}{\llangle \mu_J^2\rrangle} \ar[d,"\tilde\rho"] \ar[r] & \Z/2\Z \ar[r] & 1 \\
& & SO(3) & &
\end{tikzcd} \]
in which the top row is exact, and   the image of $r$ is an odd-order cyclic subgroup of $SO(3)$.  We will write $\img(r) = \langle x\rangle$, where $x \in SO(3)$ is an element of some odd order $n \geq 1$ and $n$ divides $|H_1(\dcover(J))|$.

Every element of $\pi_1(L\setminus N(J))/\llangle \mu_J^2\rrangle$ has the form
\[ g \text{ or } \alpha \cdot g, \qquad g\in\pi_1(\dcover(J)) \]
where $\alpha$ is some element of the nontrivial coset of the index-2 subgroup $\pi_1(\dcover(J))$.  Thus, if we write $y=\tilde\rho(\alpha)$, then
\[ \img(\tilde\rho) = \langle x\rangle \cup \left(y\cdot\langle x\rangle\right). \]
If $y\in\langle x\rangle$ then $\img(\tilde\rho)$ is cyclic of order $n$, but this is impossible since $\img(\tilde\rho)$ contains the element $\tilde\rho(\mu_J) = \ad(i)$ of order 2.  Thus, $\langle x\rangle \cong \Z/n\Z$ is an index-2 subgroup of $\img(\tilde\rho) \subset SO(3)$, so $\img(\tilde\rho)$ has order $2n$, which is not a multiple of 4.  Now, every finite subgroup of $SO(3)$ is either cyclic, dihedral, tetrahedral, octahedral, or icosahedral, and the latter three have order 12, 24, or 60, so $\img(\tilde\rho)$ must be cyclic or dihedral, of order $2n$.
\end{proof}

\begin{proposition} \label{prop:image-rho}
Suppose  $H_1(L)$ has odd order, that $J \subset L$ is primitive, and that $\dcover(L,J)$ is an $SU(2)$-abelian rational homology sphere.  Let
\[ \rho: \pi_1(L \setminus N(J)) \to SU(2) \]
be a representation satisfying $\rho(\mu_J)=i$.  Then $\img(\rho)$ is either cyclic or binary dihedral, and has order $4n$, where $n$ divides $|H_1(\dcover(L,J))|$.  Moreover, if \[|H_1(\dcover(L,J))| = |H_1(L)|,\] then $\rho$ has cyclic image.
\end{proposition}

\begin{proof}
By Lemma~\ref{lem:image-ad-rho}, we know that the image of the composition
\[ \pi_1(L \setminus N(J)) \xrightarrow{\rho} SU(2) \xrightarrow{\ad} SO(3) \]
is either cyclic or dihedral, and has order $2n$, where $n$ divides $|H_1(\dcover(J))|$ and is thus odd.  Note that $\img(\rho) \subset SU(2)$ contains the kernel $\{\pm1\}$ of the map \[\ad: SU(2)\to SO(3),\] since $\rho(\mu_J^2)=-1$. Therefore, \[\img(\rho) = \ad^{-1}(\img(\ad\rho))\] has order $4n$.  If $\img(\ad\rho) \subset SO(3)$ is cyclic then so is $\img(\rho) \subset SU(2)$, whereas if $\img(\ad\rho) \subset SO(3)$ is dihedral then $\img(\rho) \subset SU(2)$ is binary dihedral.

For the last assertion, if $|H_1(\dcover(L,J))| = |H_1(L)|$ then Proposition~\ref{prop:h1-bdc-equals-h1-l} says that $\img(\rho)$ cannot be non-abelian and binary dihedral, so it must be cyclic.
\end{proof}

The instanton knot homology $\KHI(L,J)$ is half of an instanton homology group defined as the Morse homology of a Chern--Simons functional on an associated space of connections. As discussed in  \cite[\S7]{km-excision}, the space of critical points of this functional can be identified with a double cover of the representation variety \[R(J) = \{\rho:\pi_1(L\setminus N(J))\to SU(2) \mid \rho(\mu_J) = i\}.\] If every element of $R(J)$ has cyclic image, then this is a trivial double cover. If in addition each $\rho\in R(J)$ is nondegenerate (corresponds to  nondegenerate critical points of the Chern--Simons functional in a suitable sense), then it follows that \begin{equation}\label{eqn:ineqdim}\dim_\C \KHI(L,J) \leq |R(J)|.\end{equation} 
We would thus like to know when each of the cyclic representations $\rho\in R(J)$ is nondegenerate.
This was addressed in the proof of \cite[Theorem~4.8]{sz-pillowcase}:

\begin{lemma} \label{lem:betti-1-implies-nondegenerate}
Let \[\rho: \pi_1(L \setminus N(J)) \to SU(2)\] be a representation with cyclic image, satisfying $\rho(\mu_J) = i$.  If \[\dim_\R H^1(L\setminus N(J);\ad(\rho))=1,\] then $\rho$ is nondegenerate as a generator of $\KHI(L,J)$.
\end{lemma}

\begin{proof}
The proof of \cite[Theorem~4.8]{sz-pillowcase} establishes this for knots in $S^3$, in the further situation where we have introduced some holonomy perturbation $\Phi$ and replaced the condition $\rho(\mu_J)=i$ with a more general condition described by some curve $C'$ in the pillowcase orbifold (i.e., the $SU(2)$ character variety of $T^2$).  The current setting is simpler, though: we take $\Phi=0$ and the original condition $\rho(\mu_J)=i$, and then the proof works verbatim.
\end{proof}

\begin{proposition} \label{prop:branched-cover-nondegenerate}
Suppose  $H_1(L)$ has odd order and   $J \subset L$ is  primitive.  Let
\[ \rho: \pi_1(L \setminus N(J)) \to SU(2) \]
be a representation with cyclic image, satisfying $\rho(\mu_J) = i$.  If the $2|H_1(L)|$-fold cyclic cover of $L \setminus N(J)$ has first Betti number 1, then $\rho$ is nondegenerate as a generator of $\KHI(L,J)$.
\end{proposition}

\begin{proof}
Let $C\subset Z$ be as above, so  that $[\mu_C]$ generates $H_1(L\setminus N(J)) \cong \Z$ and \[[\mu_J] = |H_1(L)|\cdot [\mu_C].\]  Since $\rho$ has cyclic image, it follows that $\ad\rho$ has cyclic image as well and therefore factors through $H_1(L\setminus N(J))$.  Moreover, $\img(\ad\rho)$ has order $n=2r$ for some divisor $r$ of $|H_1(L)|$, because it sends $|H_1(L)| \cdot [\mu_C] = [\mu_J]$ to the order-2 element $\ad(i)$.

For each divisor $d$ of $n$, let $Y_d$ be the $d$-fold cyclic cover of $L \setminus N(J)$ corresponding to the index-$d$ subgroup
\[ \ker\left(\pi_1(L\setminus N(J)) \xrightarrow{\ad\rho} SO(3) \xrightarrow{x \mapsto x^{n/d}} SO(3)\right). \]
Boyer and Nicas \cite[Theorem~1.1 and Remark~1.2]{boyer-nicas} proved that
\[ \dim H^1(L \setminus N(J); \ad\rho) = \dim H^1(L\setminus N(J)) + \frac{2}{\varphi(n)} \sum_{d\mid n} \mu\left(\frac{n}{d}\right) b_1(Y_d), \]
where $\mu$ denotes the M\"obius function.  If we can show that $b_1(Y_d) = 1$ for all $d \mid n$, then the sum on the right will be
\[ \sum_{d \mid n} \mu\left(\frac{n}{d}\right) = 0 \]
since $n \geq 2$, and so we will have $\dim H^1(L\setminus N(J); \ad\rho) = 1$ as desired.

Note for each $d \mid n$ that $Y_n$ is a finite cover of $Y_d$, which in turn is a finite cover of $L \setminus N(J)$.  It follows by a transfer argument that their first Betti numbers satisfy
\[ b_1(Y_n) \geq b_1(Y_d) \geq b_1(L\setminus N(J)) = 1, \]
so if $b_1(Y_n) = 1$ then we will have $b_1(Y_d)=1$ for all $d\mid n$ after all.  Similarly, since $n$ divides $2|H_1(L)|$, we conclude that if the $2|H_1(L)|$-fold cyclic cover of $L\setminus N(J)$ has $b_1=1$ then we also have $b_1(Y_n)=1$, and hence $\rho$ is nondegenerate.
\end{proof}

We now put together the various results of this section to prove Theorem \ref{thm:floer-simple}.

\begin{proof}[Proof of Theorem~\ref{thm:floer-simple}]
Proposition~\ref{prop:image-rho} says that every representation
\[ \rho: \pi_1(L \setminus N(J)) \to SU(2) \]
with $\rho(\mu_J) = i$ has  cyclic image.  By Proposition~\ref{prop:h1-bdc-equals-h1-l} there are $|H_1(L)|$ such $\rho$. Thus, if they are all nondegenerate, then we will have
\[ \dim_\C \KHI(L,J) \leq |R(J)|=|H_1(L)|, \] as in \eqref{eqn:ineqdim}.
In the opposite direction, we know that
\[  |H_1(L)|\leq \dim_\C I^\#(L) \leq \dim_\C \KHI(L,J) \]
by \cite[Theorem~1.2]{li-ye-sutures}, so equality will follow.  It remains to establish the nondegeneracy of these cyclic representations; by Proposition~\ref{prop:branched-cover-nondegenerate} it will suffice to show that the $2|H_1(L)|$-fold cyclic cover of $L\setminus N(J)$ has first Betti number $1$.

Let $Y$ be the $n$-fold cyclic cover of $L\setminus N(J) \cong Z \setminus N(C)$, where $n=2|H_1(L)|$.  We can Dehn fill $Y$ along a lift $\tilde\mu$ of $\mu_C^n$ to get the $n$-fold branched cyclic cover of $C$, and this Dehn filling decreases the first Betti number by $1$ since $[\tilde\mu]$ is a non-torsion element of $H_1(Y)$.  (Indeed, we can lift a Seifert surface for $C$ to $Y$ to get a surface whose intersection pairing with $\tilde\mu$ is $1$.)  It thus follows that
\[ b_1(Y)=1 \text{ if and only if } b_1(\Sigma_n(Z,C)) = 0. \]
But a theorem of Fox \cite[Theorem~8.21]{burde-zieschang} says that
\[ |H_1(\Sigma_n(Z,C))| = \left|\prod_{k=1}^{n-1} \Delta_C(e^{2\pi i k/n})\right|, \] where $\Sigma_n(Z,C)$ is the $n$-fold cyclic branched cover of $C\subset Z$.
It follows that $b_1(\Sigma_n(Z,C))=0$ if and only if no nontrivial $n$th root of unity is a zero of the Alexander polynomial $\Delta_C(t)$.

We now wish to show that the cyclotomic polynomials $\Phi_d(t)$ do not divide $\Delta_C(t)$ for all divisors $d \geq 2$ of $n=2|H_1(L)|$.  If $|H_1(L)|$ is an odd prime power, say $p^e$, then either
\begin{itemize}
\item $d=2$, and $\Phi_2(t)=1+t$ implies that $\Phi_d(1) = 2$;
\item $d=p^k$ with $1 \leq k \leq e$, and then
\[ \Phi_{p^k}(t) = 1 + t^{p^k} + t^{2\cdot p^k} + \dots + t^{(p-1)\cdot p^k} \]
implies that $\Phi_d(1) = p$; or
\item $d=2\cdot p^k$ with $1 \leq k \leq e$, and
\[ \Phi_{2\cdot p^k}(t) = \Phi_{p^k}(-t) = 1 - t^{p^k} + t^{2\cdot p^k} - \dots + t^{(p-1)\cdot p^k} \]
implies that $\Phi_d(-1) = p$.
\end{itemize}
Thus, if some $\Phi_d(t)$ divides $\Delta_C(t)$ then at least one of $|\Delta_C(1)|$ and $|\Delta_C(-1)|$ will be strictly greater than $1$.  But we know that $\Delta_C(1)=1$ for all knots in homology spheres, and we know from Lemma~\ref{lem:h1-bdc-alex}
that
\[ |\Delta_C(-1)| = |H_1(\dcover(Z,C))| = 1, \]
so we cannot have $\Phi_d(t) \mid \Delta_C(t)$ for any divisor $d \geq 2$ of $2p^e$.  We conclude that
\[ b_1(\Sigma_n(Z,C)) = 0, \]
proving the required nondegeneracy.
\end{proof}

\section{Floer simple knots  and  Theorem \ref{thm:main}}\label{sec:l32}

Recall that a knot in a lens space $L$ is \emph{simple} if it is isotopic to a union of two arcs, each contained in a meridional disk for one of the two solid tori in a genus-1 Heegaard splitting of $L$. For instance, the core of each solid torus in such a splitting is a simple knot. Moreover, there is a unique (oriented) simple knot in every homology class in $H_1(L)$. See \cite[\S2.1]{rasmussen-lens-space-surgeries} and Section~\ref{sec:other} for more on simple knots.

Our main result in this section is Theorem \ref{thm:classify-simple-knots} below. This theorem is substantially more general than is needed for the proof of Theorem \ref{thm:main}, but we will use this generality in the next section to extend the latter theorem to additional slopes in the interval $[3,5)$.

\begin{theorem} \label{thm:classify-simple-knots}
Suppose $J$ is a primitive (oriented) knot in a lens space $L$ such that
\[\dim_\C \KHI(L,J)=|H_1(L)|.\]
Then $J$ is rationally fibered of the same genus as the simple knot $S$ in its homology class. Moreover, the symmetrized Alexander polynomials of $J$ and $S$ agree, \[\Delta_J(t) = \Delta_S(t). \] Finally, if \[g(J)  \leq \frac{|H_1(L)|+1}{4}\] then $J$ is isotopic to $S$.
\end{theorem}

Before proving this theorem, we use it to prove our main result, Theorem \ref{thm:main}.

\begin{proof}[Proof of Theorem \ref{thm:main}] Suppose, for a contradiction, that $S^3_3(K)$ is $SU(2)$-abelian for some nontrivial knot $K$.  Proposition \ref{prop:surgery-cover} implies that this surgered manifold is homeomorphic to a branched double cover of the lens space \[L=L(3,2)\] along a knot $J$ which is homologous to 5 times a core in a genus-1 Heegaard splitting of $L$. In particular,  the homology class of $J$ (with any orientation) generates \[H_1(L)\cong \Z/3\Z\] since $\gcd(3,5) = 1,$ meaning that $J\subset L$ is primitive. Its branched double cover is therefore unique, by Lemma \ref{lem:h1-bdc}, so we may write \[S^3_3(K) \cong \Sigma(L,J).\] Since \[|H_1(\Sigma(L,J))| = |H_1(L)| = 3,\] Theorem \ref{thm:floer-simple} implies that  \[\dim_\C \KHI(L,J)=|H_1(L)|=3.\] Then Theorem \ref{thm:classify-simple-knots} says that $J$ is rationally fibered and has the same genus as the simple knot in its homology class (for any orientation on $J$). But each primitive simple
knot in $L$ is homologous to, and hence isotopic to, a core of a solid torus in a  genus-1 Heegaard splitting of $L$, and therefore  has  genus 0. It follows that $J$ has genus 0 as well. The fibered exterior of $J$ is then a solid torus, which implies that its branched double cover $\dcover(L,J)$ is a lens space. But Proposition \ref{prop:surgery-cover} says that $S^3_3(K)$ is not Seifert fibered, a contradiction.\end{proof}

It remains  to prove Theorem \ref{thm:classify-simple-knots}. Before doing so, we review some facts about instanton knot homology and Heegaard knot Floer homology.
We will assume for the discussion below that $J$ is a primitive knot in an irreducible rational homology sphere $L$ with $|H_1(L)|=p > 1$. In particular, the complement of $J$ is irreducible.

As in \S\ref{sec:floer-simple}, we can identify $L\setminus N(J)$ with $Z \setminus N(C)$ for a knot $C$ in some homology sphere $Z$. 
Let
\[ (F,\partial F) \subset (L\setminus N(J), \partial N(J)) \]
be a genus-minimizing Seifert surface for $C$, so that $F$ generates $H_2(L\setminus N(J), \partial N(J))$.  Then $\partial F$ is a Seifert longitude for $C$. Since \[[\mu_J]  = |H_1(L)|\cdot [\mu_C] = p[\mu_C] \in H_1(L\setminus N(J)) \] (see \S\ref{sec:floer-simple}), $\partial F$ represents a primitive class on $\partial N(J)$, and has algebraic intersection number $\pm p$ with the meridians of $J$.

Let $\gamma_{\mu}$ be the disjoint union of  two oppositely-oriented meridians of $J$ on $\partial N(J)$. Then $(L\setminus N(J),\gamma_{\mu})$ is a balanced sutured manifold, and its sutured instanton homology  defines the instanton knot homology of $J\subset L$,   \[ \KHI(L,J):=\SHI(L\setminus N(J),\gamma_{\mu}), \] as in \cite[Definition 7.13]{km-excision}. Let us arrange that $\partial F$ intersects $\gamma_{\mu}$ in $2p$ points. Then the construction in  \cite[\S 2.3]{li-ye-Euler} defines an integer-valued Alexander grading on $\KHI(L,J)$,\[\KHI(L,J) = \bigoplus_{i \in \Z} \KHI(L,J,i).\] We will review below how this Alexander grading detects genus and fiberedness, and yields a graded Euler characteristic which agrees with that of Heegaard knot Floer homology.

Let \[n := \frac{1}{2}(p-\chi(F)) = g(F) + \frac{p-1}{2}. \] Then \cite[Theorem~2.30]{li-ye-Euler} says that  \[ \KHI(L,J,i)=0 \text{ for }|i|>n,\] and that if $(M_\pm,\gamma_\pm)$  are the sutured manifolds obtained by decomposing $(L\setminus N(J),\gamma_{\mu})$ along $\pm F$, then \[ \KHI(L,J,\pm n) \cong \SHI(M_\pm,\gamma_\pm).\] Since $L\setminus N(J)$ is irreducible with toroidal boundary, and $F$ has minimal genus and minimal intersection with $\gamma_{\mu}$, the sutured manifolds $(M_\pm,\gamma_\pm)$  are both taut.  Their sutured instanton homologies are therefore nontrivial \cite[Theorem~7.12]{km-excision}; that is, \[\KHI(L,J,\pm n) \neq 0.\] In this way, instanton knot homology detects the genus of $J$. Note, moreover, that $M_+ = M_-$ and $\gamma_+ = -\gamma_-$. This implies by \cite[Theorem~2.30]{li-ye-Euler} that the sutured instanton homologies of $(M_\pm,\gamma_\pm)$ are isomorphic, \[\KHI(L,J, n) \cong \KHI(L,J, -n).\]
The fact that \[H_2(L \setminus N(J)) \cong H_2(Z \setminus N(C)) =0\] implies that 
$ H_2(M_\pm)  = 0 $ as well. We can thus apply \cite[Theorem~1.2]{gl-decomposition} to conclude that \[\dim_\C\KHI(L,J,\pm n) = 1\] iff $(M_\pm,\gamma_\pm)$ are product sutured manifolds, or, equivalently, iff $J$ is a rationally fibered knot with fiber surface $F$. In this way, instanton knot homology detects whether $J$ is rationally fibered. Heegaard knot Floer homology detects genus and fiberedness in  the same way.

There is a $\Z/2\Z$-grading on instanton knot homology which descends to a $\Z/2\Z$-grading on each Alexander-graded summand $\KHI(L,J,i)$,
so we can define the graded Euler characteristic of $\KHI(L,J)$ by
\[ \chigr(\KHI(L,J)) = \sum_{i\in\mathbb{Z}}\chi(\KHI(L,J,i))\cdot t^i. \]
Recent work of Li and Ye  \cite[Theorem 1.2]{li-ye-Euler} relates this with the graded Euler characteristic of Heegaard knot Floer homology. Namely, they prove that \[ \chigr(\KHI(L,J)) = \pm t^s \cdot \chigr(\HFK(L,J)),\] 
for some $s \in \Z$, where the latter is defined with respect to the corresponding  Alexander and $\Z/2\Z$-gradings in the Heegaard  Floer setting. 

Recall that  \begin{equation}\label{eqn:positivity} \chigr(\HFK(L,J))_{t=1}= \chi(\HFhat(L)) = |H_1(L)|=p.\end{equation}
Recall further that the Alexander grading on the Heegaard knot Floer homology of $J$ can be viewed as a grading by   relative $\spc$ structures in $\spc(L,J)$. In the spectral sequence  \[\HFK(L,J)\implies \HFhat(L),\] this grading projects to the $\spc$ grading on the latter via the natural map \[\spc(L,J)\to \spc(L).\] These sets of $\spc$ structures are affine spaces over the corresponding first homology groups, and the projection map above can be identified with the  quotient map \[H_1(L\setminus N(J))\,\cong\, \Z \to \Z/p\Z\, \cong \,H_1(L).\] More concretely, there is an identification $\spc(L)\cong \Z/p\Z$ such that the spectral sequence above restricts, for each $[i]\in \Z/p\Z$, to a spectral sequence  \begin{equation}\label{eqn:hfk}\bigoplus_{m\equiv i\!\!\!\!\!\pmod{p}}\HFK(L,J,m)\implies \HFhat(L,[i]).\end{equation}  In particular, for each $[i]\in\Z/p\Z$, we have  \begin{equation}\label{eqn:euler-1}\sum_{m\equiv i\!\!\!\!\!\pmod{p}}\chi(\HFK(L,J,m))= \chi(\HFhat(L,[i]))=1.\end{equation}
The  following additional constraints  on the graded Euler characteristic of Heegaard knot Floer homology will be important in our proof of Theorem \ref{thm:classify-simple-knots}. 

The first lemma below is proved in \cite[Corollary 5.3]{rasmussen-lens-space-surgeries}.
\begin{lemma}
\label{lem:difference}
Let $J_1$ and $J_2$ be two homologous primitive  knots in $L$.  Then 
the difference \[\chigr(\HFK(L,J_1))-\chigr(\HFK(L,J_2))\] is divisible by $(t^p-1)^2$.
\end{lemma}

The next lemma follows immediately from \eqref{eqn:euler-1}.

\begin{lemma}\label{lem:euler}
Let $J$ be a primitive knot in $L$. Then, for each $[i]\in \Z/p\Z$,  we have that
\[\chi(\HFK(L,J,m)) \neq 0 \] for some integer $m\equiv i\pmod{p}$.
\end{lemma} 
The final lemma below is proved in \cite[Proof of Theorem 1]{rasmussen-lens-space-surgeries}. We  only need it for the last claim in Theorem \ref{thm:classify-simple-knots} in the case where $g(J)>0$. We do not need it when $g(J)=0$ (in particular, we do not need it for the proof of Theorem \ref{thm:main}) because we know automatically in this case  that $J$ is  a core and thus simple.

\begin{lemma}\label{lem:genus-small}
Let $J$ be a primitive knot in $L$. Recall that we have assumed that $|H_1(L)|=p$.
Suppose in addition that $\dim_{\Z/2\Z}\HFhat(L)=p$ and \[g(J)  < \frac{p+1}{2}.\] Then $\dim_{\Z/2\Z}\HFK(L,J)=p$. That is, $J$ is Heegaard Floer simple.
\end{lemma}


We may now prove Theorem \ref{thm:classify-simple-knots}.

\begin{proof}[Proof of Theorem \ref{thm:classify-simple-knots}] Let $J\subset L$ be as in the hypothesis of the theorem, and let $S$ denote the simple knot in the same homology class as $J$. Simple knots admit Heegaard diagrams for which the  knot Floer complex has trivial differential; in particular, they are Heegaard Floer simple.\footnote{Simple knots are also instanton Floer simple, by \cite{li-ye-sutures}; see also \cite{bly}.} Thus, \[\dim_{\Z/2\Z}\HFK(L,S) = p.\] It then follows  from Lemma \ref{lem:euler} and  \eqref{eqn:positivity} that $S$ satisfies the property\\
\begin{equation}
\begin{array}{l}
\text{for each } [i]\in\Z/p\Z, \text{ there is a \emph{unique} integer }m\equiv i\!\!\!\pmod{p} \text{ such that:} \\
 \chi(\HFK(L,S,m))=1, \textrm{ and } \chi(\HFK(L,S,k))=0 \textrm{ for all other } k\equiv i\!\!\!\pmod{p}.

\end{array}
\tag{$\ast$}\label{eq:property-ast}
\end{equation}\\
It also follows from Lemma \ref{lem:euler} that  \[\dim_{\Z/2\Z}\HFK(L,S,m) = 0 \textrm{ or } 1\] for each $m\in \Z$.
Similarly, Lemma \ref{lem:euler} applied to $J$, combined with the facts that \[\dim_{\C}\KHI(L,J) = p\] and that \begin{equation}\label{eqn:s}\chigr(\KHI(L,J)) = \pm t^s \cdot \chigr(\HFK(L,J))\end{equation} for some $s\in \Z$, implies that $J$ also satisfies \eqref{eq:property-ast}, and, additionally, that \[\dim_{\C}\KHI(L,J,m) = 0 \textrm{ or } 1\] for each $m\in \Z$. In particular, \[\dim_{\C}\KHI(L,J,\pm n) = 1\] for \[n = g(J) + \frac{p-1}{2},\] which implies that $J$ is rationally fibered. Since \[\KHI(L,J, i)=0\] for $|i|>n$, and likewise for the Heegaard knot Floer homology of $J$, and since  $\chigr(\HFK(L,J))$ is symmetric under the substitution $t\mapsto t^{-1}$, this then implies that $s=0$ in \eqref{eqn:s}. 

We next prove that $J$ and $S$ have the same genus. For this, we  claim, as in \cite[Proof of Theorem 2]{rasmussen-lens-space-surgeries}, that  \begin{equation}\label{eqn:euler-same}\overline\Delta_J(t):=\chigr(\HFK(L,J)) = \chigr(\HFK(L,S))=:\overline\Delta_S(t).\end{equation} Indeed, the fact that $J$ and $S$ both satisfy property \eqref{eq:property-ast} means that we can write
\begin{align*}
\overline\Delta_S(t) &=\sum_{[i]\in\Z/p\Z} t^{n_{[i]}}\\
\overline\Delta_J(t) &=\sum_{[i]\in\Z/p\Z} t^{n_{[i]} + p\cdot m_{[i]}}
\end{align*}
for some integers $n_{[i]} \equiv i \pmod{p}$ and $m_{[i]}$.
Since $J$ and $S$ are in the same homology class, Lemma \ref{lem:difference} says that $(t^p-1)^2$ divides the difference
\[\overline\Delta_J(t)-\overline\Delta_S(t) = \sum_{[i]\in\Z/p\Z} t^{n_{[i]}}(t^{p\cdot m_{[i]}}-1).\] However, after dividing this difference by $t^p-1$, we are left with \[\sum_{\substack{[i]\in\Z/p\Z\\
m_{[i]} \neq 0}} t^{n_{[i]}}(1+t^{m_{[i]}} + \dots + t^{(p-1)\cdot m_{[i]}}).\] For this sum to be divisible by another factor of $t^p-1$, it must have $1$ as a root. But that is only possible if the sum is trivial---that is, if  $m_{[i]}=0$ for every $[i]\in \Z/p\Z$. This
 proves that $\overline\Delta_J(t)=\overline\Delta_S(t)$ as desired.

Combining \eqref{eqn:s} (recall that  $s=0$) with \eqref{eqn:euler-same}, we have therefore shown that \[\chigr(\KHI(L,J)) = \pm\chigr(\HFK(L,S)).\] Since the instanton Floer homology of $J$ has dimension at most 1 in each Alexander grading, and likewise for the Heegaard knot Floer homology of $S$, it follows that for each integer $i$, \[\KHI(L,J,i) = 0\, \textrm{ iff } \,\HFK(L,S,i) = 0.\] This implies that $J$ and $S$ have the same genus, since the Alexander gradings in these two theories detect genus in the same way.

For the claim about Alexander polynomials, Rasmussen proves in \cite[\S3.7]{rasmussen-lens-space-surgeries} that, for any primitive knot $J$ in a rational homology sphere $L$ of order $p$, we have\[\overline\Delta_J(t) = \Delta_J(t) \cdot \frac{t^{p/2}-t^{-p/2}}{t^{1/2}-t^{-1/2}}, \] where $\Delta_J(t)$ is the symmetrized Alexander polynomial of the complement $L\setminus N(J)$. We have shown above that $\overline\Delta_J(t) = \overline\Delta_S(t)$, from which it follows that $\Delta_J(t) = \Delta_S(t)$ as well.

For the last claim of the theorem, suppose  \begin{equation*}\label{eqn:genus-j-2}g(J)  \leq \frac{|H_1(L)|+1}{4} = \frac{p+1}{4}.\end{equation*} Then $J$ is Heegaard Floer simple, by Lemma \ref{lem:genus-small}. Furthermore, Baker proved in  \cite[Theorem 1.1]{baker-smallgenus} that any knot $J$ satisfying the  genus bound above is  a 1-bridge knot in $L$.  Finally, Hedden proved in \cite[Proposition 3.3]{hedden-simple} that any 1-bridge knot in a lens space which is also Heegaard Floer simple is  simple. Thus, $J\cong S$.
\end{proof}

\section{Other small surgeries and $SU(2)$}\label{sec:other}

As mentioned in the introduction, we suspect that  the following is true:

\begin{conjecture} \label{conj:other} If $K\subset S^3$ is a nontrivial knot, then $S^3_r(K)$ is not $SU(2)$-abelian for any rational number $r \in [0,5)$.
\end{conjecture}
 Kronheimer--Mrowka proved this conjecture  for all $r\in [0,2]$ in \cite{km-su2}. Baldwin--Sivek proved it in \cite[Theorem 1.8]{bs-lspace} for the dense subset of rationals  in $(2,3)$ given by \[r=p/q\in (2,3),\] where $p,q$ are coprime  and $p$ is a prime power. Indeed, they proved that if $S^3_r(K)$ is $SU(2)$-abelian in this case, then $K$ has genus 1, and is fibered and strongly quasipositive. The only  such knot, the right-handed trefoil, does not admit $SU(2)$-abelian surgeries in this range.

As alluded to in \S\ref{ssec:other}, proving Conjecture \ref{conj:other} is much more difficult for the corresponding dense subset of rationals in the interval $[3,5)$, given by  \[r=p/q\in [3,5), \] where $p,q$ are coprime and $p$ is a prime power. If $S^3_r(K)$ is $SU(2)$-abelian in this case, then it is  an instanton L-space, by Lemma \ref{lem:su2-cyclic-lspace}, from which it follows  that $K$ has genus at most 2, and is fibered and strongly quasipositive \cite{bs-lspace}. One  can rule out the trefoil, so  $K$ has genus 2, in which case it also has the same Alexander polynomial as $T_{2,5}$, by Lemma \ref{lem:i-sharp-basics}. The difficulty is that there are infinitely many such knots; see, for example,  \cite{misev2017families}. 

In this section, we illustrate how our techniques can   be used to prove Conjecture \ref{conj:other} for infinitely many  slopes $r\in [3,5)$, culminating in the proof of Theorem \ref{thm:other-small-surgeries}. 

We start with a comparatively easy case, in which the numerator of $r$ is a power of $2$, which  is handled with a trick involving a result of Klassen. In this case, we can actually deal with all slopes less than 7.

\begin{theorem}
\label{thm:power-2}
Let $p,q$ be coprime positive integers with $p/q<7$, where $p$ is a power of $2$.  Then $S^3_{p/q}(K)$ is not $SU(2)$-abelian for any nontrivial knot $K\subset S^3$.
\end{theorem}

\begin{proof}
Suppose $Y=S^3_{p/q}(K)$ is $SU(2)$-abelian. Then it is an instanton L-space by Lemma~\ref{lem:su2-cyclic-lspace}.  It then follows from \cite[Theorem~1.15]{bs-lspace}  that
\[ 2g(K) - 1 \leq p/q < 7, \]
which implies that $g(K) \leq 3$.  If $g(K)=1$ then $K$ is the right-handed trefoil, and $Y$ is Seifert fibered with base orbifold $S^2(2,3,\Delta)$ where
\[ \Delta = |6q-p|. \]
We cannot have $\Delta \leq 1$, because otherwise either $\frac{p}{q} = \frac{6}{1}$ or 
\[ \frac{p}{q} = 6 \pm \frac{1}{n} = \frac{6n\pm1}{n} \]
for some integer $n \geq 1$, and in these cases $p$ is not actually a power of $2$.  But then $\Delta > 1$ and so $Y$ is not $SU(2)$-abelian \cite{sz-menagerie}, a contradiction.  Therefore, $g(K)$ is either $2$ or $3$, which then implies that $p/q \geq 3$.  Since $p$ is a power of 2, it must be a multiple of 4.

Theorem~\ref{thm:genus-2-lspace} now tells us that $K$ has Alexander polynomial
\[\Delta_K(t) = t^2-t+1-t^{-1}+t^{-2}\]
if $g(K)=2$, whereas if $g(K)=3$, then the same argument as in the proof of Lemma~\ref{lem:i-sharp-basics} (appealing to \cite[Theorem~1.9]{li-ye-surgery} as well as \cite[Proposition~7.16]{km-excision} and \cite[Theorem~1.7]{bs-trefoil}) says that
\begin{align*}
\Delta_K(t) &= t^3-t^2+t-1+t^{-1}-t^{-2}+t^{-3} \\
&\phantom{=}\text{ or } t^3-t^2+1-t^{-2}+t^{-3}.
\end{align*}
In each of these cases we have
\[ \det(K) = |\Delta_K(-1)| = 5 \text{ or } 7 \text{ or }3, \]
respectively.  From here we argue exactly as in \cite[Lemma~9.5]{bs-lspace}: Klassen \cite[Theorem~10]{klassen} proved that there are \[\frac{\det(K)-1}{2} = 2\text{ or } 3 \text{ or }1,\] respectively, conjugacy classes of non-abelian representations
\[ \rho: \pi_1(S^3\setminus K) \to SU(2) \]
with image in the binary dihedral group $D_\infty = \{e^{i\theta}\} \cup \{e^{i\theta}j\}$. Let $\mu$ be a meridian of $K$ and $\lambda$ the Seifert longitude. Following the proof of \cite[Theorem~10]{klassen}, since $\rho$ has non-abelian image we must have that $\rho(\mu) = j$ up to conjugacy, and hence $\rho(\mu^4) = 1$.  Moreover, we know that $\lambda$ lies in the second commutator subgroup of $\pi_1(S^3\setminus K)$; thus $\rho(\lambda)$ lies in the second commutator subgroup of $D_\infty$, and the latter is trivial, so $\rho(\lambda)=1$.  Since $p$ is a multiple of 4, it follows that $\rho(\mu^p\lambda^q) = 1$, so $\rho$ descends to a representation
\[ \pi_1(S^3_{p/q}(K)) \to SU(2) \]
with non-abelian image, but we assumed $Y= S^3_{p/q}(K)$ is $SU(2)$-abelian, a contradiction.
\end{proof}

We next turn to the much more difficult case in which the numerator of $r$ is an odd prime power; this will require the full strength of the techniques in this paper.  To start, we  review some facts about  simple knots in lens spaces, and then study these knots in  greater depth.

If $A\subset S^3$ is an unknot, then  the lens space $L(a,b)$ is given in our notation by \[L(a,b) = S^3_{a/b}(A).\] The standard genus-1 Heegaard splitting of this lens space is given by the union of the solid torus exterior $S^3\setminus N(A)$ with the surgery solid torus. The core of each solid torus generates the first homology of the lens space. Let $S(a,b,k)$ denote the unique (oriented) simple knot which is homologous to $k$ times the core of the surgery solid torus, for some orientation of this core.\footnote{This simple knot is denoted by $K(a,b,k)$ in \cite{rasmussen-lens-space-surgeries}.} Note that \[S(a,b,k) \cong S(a,b,k')\] when $k\equiv k'\pmod{a}.$ Moreover, $S(a,a-b,k)\subset -L(a,b)$ is the orientation-reverse of the mirror of $S(a,b,k)$. Finally, $S(a,b,\pm 1)$ is isotopic to a core and therefore has genus $0$. 

Several of our results thus far, like Proposition \ref{prop:surgery-cover},  pertain to a  knot $J\subset L(p,2q)$ which is homologous to 5 times a core of the \emph{other} solid torus  $S^3\setminus N(A)$ in the genus-1 splitting of this lens space. It is easy to check (for some orientation on $J$) that \[[J] = [S(p,2q,10q)]\in H_1(L(p,2q)).\]
Note that when $p$ is odd,
the knot $S(p,2q,10q)$ is primitive if and only if $\gcd(p,5)=1$. 

The main theorem of this section is as follows:

\begin{theorem} \label{thm:simple-abelian}Let $p$ and $q$ be coprime positive integers with $p/q\in [3,5)$, where $p$ is an odd prime power and $\gcd(p,5)=1$. If either $p/q\in (4,5)$ or  \[g(S(p,2q,10q))\leq \frac{p+1}{4},\] then
$S^3_{p/q}(K)$ is not $SU(2)$-abelian for any nontrivial knot $K\subset S^3$. 
\end{theorem}

Before proving this theorem, we  establish the following proposition. In addition to helping us prove Theorem \ref{thm:simple-abelian}, this proposition  also provides the genera of the relevant simple knots $S(p,2q,10q)$, making it easy to apply Theorem \ref{thm:simple-abelian} in practice. Indeed, we will apply this genus formula at the end to prove Theorem \ref{thm:other-small-surgeries}.

\begin{proposition}
\label{prop:genus-SF}
Let $p$ and $q$ be coprime positive integers with  $p/q\in [3,6]$, where $p$ is odd and $\gcd(p,5) = 1$. Then the branched double cover of \[S=S(p,2q,10q)\subset L(p,2q)=L\] is Seifert fibered with \begin{equation*} \label{eq:simple-genus}
|H_1(\dcover(L,S))| = \begin{cases}
p, & p/q\in [3,4) \\
5p, & p/q \in (4,6).
\end{cases}
\end{equation*} Moreover, this primitive simple knot has genus \[g(S) = \big|\,|4p-10\min\{2q,p-2q\}|-2\,\big|.\] 
\end{proposition}

\begin{proof}
Let $a,b$ be coprime positive integers with $b/a\in[1/3, 2/3]$ and $\gcd(a,5)=1$; eventually we will take $(a,b)=(p,2q)$. We will attempt to understand the branched double cover and genus of  the primitive simple knot \[S(a,b,5b) \subset L(a,b).\] Without loss of generality, we can assume that \[b/a \in [1/3,1/2]\] Indeed, if $b/a\in [1/2,2/3]$ then $(a-b)/a\in [1/3,1/2],$ and  we can just consider $S(a,a-b,5b)$ instead, as it is  the  orientation-reverse of the mirror of $S(a,b,5b)$.

In \cite[\S 8]{ye-constrained}, Ye determines  when a simple knot in $L(a,b)$ is isotopic to a torus knot on the Heegaard torus in the genus-1 splitting of the lens space. More precisely, let us think of $L(a,b)$ as obtained via $(a/b)$-surgery on the unknot $A\subset S^3$, with genus-1 Heegaard splitting given by the union of \[H_\alpha = S^3\setminus N(A)\] with the surgery solid torus $H_\beta$, as above. In particular, the meridian of $H_\beta$ is glued along the curve $a\mu + b\lambda$,
with respect to a meridian $\mu$ and Seifert longitude $\lambda$ of $A$ on $\partial H_\alpha$, with $\mu\cdot\lambda = -1$. Let \[\Sigma = \partial H_\alpha = -\partial H_\beta\] be the Heegaard torus in this splitting. Let $T$ be the $(5,2)$-torus knot on $\Sigma$---that is, the simple closed curve on  $\Sigma$ which is homologous in $\Sigma$
to $5\mu + 2\lambda$. Since $b/a\in [1/2,1/3]$, \cite[Theorem 8.2]{ye-constrained} says that
\[S(a,b,5b)\cong T.\] Indeed, following the discussion above, it is easy to see that \[[S(a,b,5b)]=[T]\in H_1(L(a,b)),\] so it suffices to show that $T$ is simple. The   basic idea  behind Ye's result, inspired by \cite[\S 3.1]{greene-lewallen-vafaee}, is  that if we fix a Heegaard diagram $(\Sigma,\alpha,\beta)$ for the Heegaard splitting \[L(a,b) = H_\alpha \cup_\Sigma -H_\beta\] where $\alpha$ is homologous to $\lambda$ and  $\beta$ is homologous to $a\mu + b\lambda$,  then  the slope $5/2$ is sufficiently close to $a/b$ in this case that we can represent $T$ as the union of an arc in $\beta$ with an arc in $\alpha$.  Thus, $T$ is simple.

Our goals are then  to show that the branched double cover of $T$ is Seifert fibered,   compute its genus, and  compute the order of its branched double cover.

For the former,   note that one can Seifert fiber the solid tori $H_\alpha$ and $H_\beta$ so that $T$ is a regular fiber in both. These glue  to give a Seifert fibration of $L(a,b)$ with two singular fibers in which $T$ is a regular fiber. The complement of a neighborhood of $T$ is thus Seifert fibered, so the double cover of this complement is as well (each circle fiber lifts to one or two circles). The Seifert fibration on this double cover then extends across the solid torus we glue in to form the branched double cover of $T$. (This is the same reasoning that shows that torus knots in $S^3$ have Seifert fibered branched double covers.)

We turn next to the genus calculation.
The  knot Floer homology of $T$, and therefore the genus of this knot, is determined by its graded Euler characteristic since simple knots are  Floer simple. Let $L = L(a,b).$ Following Rasmussen in \cite[\S3.7]{rasmussen-lens-space-surgeries}, we have that \[\chigr(\HFK(L,T)) = \Delta_T(t) \cdot \frac{t^{a/2}-t^{-a/2}}{t^{1/2}-t^{-1/2}},\] where $\Delta_T(t)$ is the symmetrized Alexander polynomial of $T$. Letting $\deg(\Delta_T(t))$ denote the top degree of this Laurent polynomial, it follows that the top non-zero Alexander grading of $\HFK(L,T)$ is given by \[\deg(\Delta_T(t)) + \frac{a-1}{2}.\] On the other hand, from the discussion in the previous section, this top Alexander grading is also given by \[g(T) + \frac{a-1}{2}.\] Therefore, \[g(T) = \deg(\Delta_T(t)).\] It thus suffices to compute this Alexander polynomial. 

The Alexander polynomial $\Delta_T(t)$ can be computed via Fox calculus, from a presentation of the fundamental group of $L\setminus N(T)$, as described in \cite[\S3.7]{rasmussen-lens-space-surgeries}. One  finds a 2-generator, 1-relator presentation of this  group using Seifert--Van Kampen, exactly as for torus knots in $S^3$. Namely, the complement of $T$ is obtained by gluing the solid tori $H_\alpha$ and $H_\beta$ together along the annulus on $\Sigma$ given by \[C=\Sigma\setminus N(T).\] Let $x$ and $y$ be the homotopy classes of the cores of $H_\alpha$ and $H_\beta$, respectively. The core of $C$ is homotopic to $x^5$ in $\pi_1(H_\alpha)$ and $y^d$ in $\pi_1(H_\beta)$ (for some orientation of these cores), where \[d = |\Delta(5/2, a/b)| = |2a-5b|.\] Therefore,  we have a presentation \[\pi_1(L\setminus N(T)) = \langle x,y \mid x^5 = y^{d}\rangle.\] This is the same as the fundamental group of the complement of the torus knot
\[T_{5, d}\subset S^3,\] so Fox calculus tells us that $\Delta_T(t)$ agrees with the Alexander polynomial of this torus knot. In particular, \[g(T) = \deg(\Delta_T(t)) = \deg(\Delta_{T_{5,d}}(t)) = |2(d-1)|= \big|\,|4a-10b|-2\,\big|.\] Recall that if $b/a\in [1/2,2/3]$ then we replace $b$ with $a-b$ in this formula. In other words, we have that \[g(S(a,b,5b)) =  \big|\,|4a-10\min\{b,a-b\}|-2\,\big|.\] If $p,q$ are as in the hypotheses of this proposition, then letting $a=p$ and $b=2q$ yields the desired genus formula. (Note that $p$ and $2q$ are coprime, since $p$ and $q$ are  and $p$ is odd.)

Finally, we compute the order of the branched double cover of \[S=S(p,2q,10q)\subset L(p,2q)=L,\] where $p$ and $q$ are as in the hypothesis of the proposition. By Lemma \ref{lem:h1-bdc-alex}, we have that \begin{equation}\label{eqn:order-k} |H_1(\dcover(L,S))|=p\cdot |\Delta_S(-1)|. \end{equation}Recall that \[\Delta_S(t) = \Delta_{T_{5,d}}(t),\] for some $d$, as explained above. If $p/q\in [3,4)$, then $2q/p\in (1/2,2/3]$,
and  \[d = |2p-5(p-2q)|= |10q-3p|\]  is odd since $p$ is odd. In this case, an easy calculation shows that $|\Delta_{T_{5,d}}(-1)| =1$,
which implies that $|H_1(\dcover(L,S))|=p$ by \eqref{eqn:order-k}. If $p/q\in (4,6)$ instead, then $2q/p\in (1/3,1/2)$,
and \[d =  |2p-10q|\] is even. In this case, $|\Delta_{T_{5,d}}(-1)| =5$,
and therefore $|H_1(\dcover(L,S))|=5p$. 
\end{proof}

We remark that the formula for $g(S(p,2q,10q))$ from Proposition~\ref{prop:genus-SF} can be simplified for various values of $p/q$, as follows:
\begin{equation} \label{eq:simple-genus}
g(S(p,2q,10q)) = \begin{cases}
20q - 6p - 2, & 3 \leq p/q < 10/3 \\
6p - 20q - 2, & 10/3 < p/q < 4\\
20q - 4p - 2,& 4 < p/q < 5 \\
4p - 20q - 2, & 5 < p/q \leq 6
\end{cases}
\end{equation}
We have omitted the cases $p/q=10/3,4,5$ above because of the requirement that $p$ be odd and not a multiple of 5.

\begin{proof}[Proof of Theorem \ref{thm:simple-abelian}]
Suppose  $Y=S^3_{p/q}(K)$ is  $SU(2)$-abelian for $p$ and $q$ as in the theorem.  Proposition \ref{prop:surgery-cover} implies that \[Y\cong \dcover(L,J)\] for some primitive knot $J$ in the lens space $L = L(p,2q),$ where  $J$ (for some orientation) is homologous  to the simple knot \[S=S(p,2q,10q),\] as described above. Proposition \ref{prop:surgery-cover} also implies that $Y$ is not Seifert fibered.
From here, Theorem \ref{thm:floer-simple} tells us that \[\dim_\C \KHI(L,J)=|H_1(L)|,\] and so Theorem \ref{thm:classify-simple-knots} implies that $\Delta_J(t) = \Delta_S(t)\textrm{ and }g(J) = g(S).$

If $p/q\in (4,5)$, then \[|H_1(\dcover(L,J))| = p\cdot |\Delta_J(-1)| = p\cdot |\Delta_S(-1)| = |H_1(\dcover(L,S))|,\] by Lemma \ref{lem:h1-bdc-alex}. But the latter  order is  $5p$, by Proposition \ref{prop:genus-SF}. This is a contradiction, since $Y = \dcover(L,J)$ has order $p$ first homology, as $(p/q)$-surgery on a knot in $S^3$.

Finally, if \[g(S)\leq \frac{p+1}{4},\] then the same is true of the genus of $J$, in which case
$J$ is isotopic to $S,$ by Theorem \ref{thm:classify-simple-knots}. But then $Y\cong\dcover(L,J)  \cong \Sigma(L,S)$ is Seifert fibered, by Proposition \ref{prop:genus-SF}, a contradiction.
\end{proof}

Finally, we combine Theorems \ref{thm:power-2} and \ref{thm:simple-abelian} and Proposition \ref{prop:genus-SF} to prove Theorem \ref{thm:other-small-surgeries}.

\begin{proof}[Proof of Theorem \ref{thm:other-small-surgeries}]
Let $p$ and $q$ be coprime positive integers with $p/q\in [3,7)$, where  $p$ is a prime power, as in the hypothesis of the theorem, and let $K\subset S^3$ be a nontrivial knot. We know that \[Y=S^3_{p/q}(K)\] is not $SU(2)$-abelian when $p$ is even, by Theorem \ref{thm:power-2}.

Suppose  now that  $p$ is odd with $\gcd(p,5) = 1$. Let us assume first that $p/q\in [4,5)$. Then Theorem \ref{thm:simple-abelian} says that $Y$ is not $SU(2)$-abelian.  Next, assume that $p/q\in [3,4)$ and \[23p-9 \leq 80q \leq 25p+9.\] These inequalities are equivalent to the statement that  both $20q-6p-2 \textrm{ and } 6p-20q-2$ are  at most $(p+1)/4$. Thus, \[g(S(p,2q,10q))\leq \frac{p+1}{4},\] by \eqref{eq:simple-genus}. Therefore, Theorem \ref{thm:simple-abelian} says that $Y$ is not $SU(2)$-abelian.
\end{proof}

\bibliographystyle{alpha}
\bibliography{References}

\begin{thebibliography}{LPCZ21}

\bibitem[Bak06]{baker-smallgenus}
Kenneth~L. Baker.
\newblock Small genus knots in lens spaces have small bridge number.
\newblock {\em Algebr. Geom. Topol.}, 6:1519--1621, 2006.

\bibitem[BF08]{boden-friedl-1}
Hans~U. Boden and Stefan Friedl.
\newblock Metabelian {${\rm SL}(n,\Bbb C)$} representations of knot groups.
\newblock {\em Pacific J. Math.}, 238(1):7--25, 2008.

\bibitem[BHS21]{bhs-cinquefoil}
John~A. Baldwin, Ying Hu, and Steven Sivek.
\newblock Khovanov homology and the cinquefoil.
\newblock arXiv:2105.12102, 2021.

\bibitem[BLY20]{bly}
John~A. Baldwin, Zhenkun Li, and Fan Ye.
\newblock {Sutured instanton homology and Heegaard diagrams}.
\newblock arXiv:2011.09424, 2020.

\bibitem[BN90]{boyer-nicas}
S.~Boyer and A.~Nicas.
\newblock Varieties of group representations and {C}asson's invariant for
  rational homology {$3$}-spheres.
\newblock {\em Trans. Amer. Math. Soc.}, 322(2):507--522, 1990.

\bibitem[BS18a]{bs-trefoil}
John~A. Baldwin and Steven Sivek.
\newblock Khovanov homology detects the trefoils.
\newblock arXiv:1801.07634, 2018.

\bibitem[BS18b]{bs-stein}
John~A. Baldwin and Steven Sivek.
\newblock Stein fillings and {SU}(2) representations.
\newblock {\em Geom. Topol.}, 22(7):4307--4380, 2018.

\bibitem[BS19]{bs-lspace}
John~A. Baldwin and Steven Sivek.
\newblock Instantons and {L}-space surgeries.
\newblock arXiv:1910.13374, 2019.

\bibitem[BS20]{bs-concordance}
John~A. Baldwin and Steven Sivek.
\newblock Framed instanton homology and concordance.
\newblock arXiv:2004.08699, 2020.

\bibitem[BS21]{bs-toroidal}
John~A. Baldwin and Steven Sivek.
\newblock {Instanton L-spaces and splicing}.
\newblock arXiv:2103.08087, 2021.

\bibitem[BZ03]{burde-zieschang}
Gerhard Burde and Heiner Zieschang.
\newblock {\em Knots}, volume~5 of {\em De Gruyter Studies in Mathematics}.
\newblock Walter de Gruyter \& Co., Berlin, second edition, 2003.

\bibitem[CC09]{cotton-clay}
Andrew Cotton-Clay.
\newblock Symplectic {F}loer homology of area-preserving surface
  diffeomorphisms.
\newblock {\em Geom. Topol.}, 13(5):2619--2674, 2009.

\bibitem[Don02]{donaldson-book}
S.~K. Donaldson.
\newblock {\em Floer homology groups in {Y}ang-{M}ills theory}, volume 147 of
  {\em Cambridge Tracts in Mathematics}.
\newblock Cambridge University Press, Cambridge, 2002.
\newblock With the assistance of M. Furuta and D. Kotschick.

\bibitem[DS94]{dostoglou-salamon}
Stamatis Dostoglou and Dietmar~A. Salamon.
\newblock Self-dual instantons and holomorphic curves.
\newblock {\em Ann. of Math. (2)}, 139(3):581--640, 1994.

\bibitem[GL19]{gl-decomposition}
Sudipta Ghosh and Zhenkun Li.
\newblock Decomposing sutured monopole and instanton {F}loer homologies.
\newblock arXiv:1910.10842v3, 2019.

\bibitem[GLV18]{greene-lewallen-vafaee}
Joshua~E. Greene, Sam Lewallen, and Faramarz Vafaee.
\newblock {(1,1) L-space knots}.
\newblock {\em Compositio Math.}, 154:918--933, 2018.

\bibitem[Hed11]{hedden-simple}
Matthew Hedden.
\newblock On {F}loer homology and the {B}erge conjecture on knots admitting
  lens spaces surgeries.
\newblock {\em Trans. Amer. Math. Soc.}, 363(2):949--968, 2011.

\bibitem[Kla91]{klassen}
Eric~Paul Klassen.
\newblock Representations of knot groups in {${\rm SU}(2)$}.
\newblock {\em Trans. Amer. Math. Soc.}, 326(2):795--828, 1991.

\bibitem[KM04a]{km-su2}
P.~B. Kronheimer and T.~S. Mrowka.
\newblock Dehn surgery, the fundamental group and {SU{$(2)$}}.
\newblock {\em Math. Res. Lett.}, 11(5-6):741--754, 2004.

\bibitem[KM04b]{km-p}
P.~B. Kronheimer and T.~S. Mrowka.
\newblock Witten's conjecture and property {P}.
\newblock {\em Geom. Topol.}, 8:295--310, 2004.

\bibitem[KM10a]{km-alexander}
Peter Kronheimer and Tom Mrowka.
\newblock Instanton {F}loer homology and the {A}lexander polynomial.
\newblock {\em Algebr. Geom. Topol.}, 10(3):1715--1738, 2010.

\bibitem[KM10b]{km-excision}
Peter Kronheimer and Tomasz Mrowka.
\newblock Knots, sutures, and excision.
\newblock {\em J. Differential Geom.}, 84(2):301--364, 2010.

\bibitem[Lim10]{lim}
Yuhan Lim.
\newblock Instanton homology and the {A}lexander polynomial.
\newblock {\em Proc. Amer. Math. Soc.}, 138(10):3759--3768, 2010.

\bibitem[LPCZ21]{lpcz}
Tye Lidman, Juanita Pinz{\'o}n-Caicedo, and Raphael Zentner.
\newblock {Toroidal homology spheres and $SU(2)$-representations}.
\newblock arXiv:2102.02621, 2021.

\bibitem[LT12]{lee-taubes}
Yi-Jen Lee and Clifford~Henry Taubes.
\newblock Periodic {F}loer homology and {S}eiberg-{W}itten-{F}loer cohomology.
\newblock {\em J. Symplectic Geom.}, 10(1):81--164, 2012.

\bibitem[LY20]{li-ye-sutures}
Zhenkun Li and Fan Ye.
\newblock Instanton {F}loer homology, sutures, and {H}eegaard diagrams.
\newblock arXiv:2010.07836, 2020.

\bibitem[LY21a]{li-ye-enhanced-Euler}
Zhenkun Li and Fan Ye.
\newblock An enhanced {E}uler characteristic of sutured instanton homology.
\newblock arXiv:2107.10490, 2021.

\bibitem[LY21b]{li-ye-Euler}
Zhenkun Li and Fan Ye.
\newblock Instanton {F}loer homology, sutures, and {E}uler characteristics.
\newblock arXiv:2101.05169, 2021.

\bibitem[LY21c]{li-ye-surgery}
Zhenkun Li and Fan Ye.
\newblock {$SU(2)$} representations and a large surgery formula.
\newblock arXiv:2107.11005, 2021.

\bibitem[Mis17]{misev2017families}
Filip Misev.
\newblock On families of fibred knots with equal {S}eifert forms.
\newblock arXiv:1703.07632, 2017.

\bibitem[Mos71]{moser}
Louise Moser.
\newblock Elementary surgery along a torus knot.
\newblock {\em Pacific J. Math.}, 38:737--745, 1971.

\bibitem[Ras07]{rasmussen-lens-space-surgeries}
Jacob Rasmussen.
\newblock {Lens space surgeries and L-space homology spheres}.
\newblock arXiv:0710.2531, 2007.

\bibitem[Sca15]{scaduto}
Christopher~W. Scaduto.
\newblock Instantons and odd {K}hovanov homology.
\newblock {\em J. Topol.}, 8(3):744--810, 2015.

\bibitem[Smi12]{smith}
Ivan Smith.
\newblock Floer cohomology and pencils of quadrics.
\newblock {\em Invent. Math.}, 189(1):149--250, 2012.

\bibitem[SZ17]{sz-pillowcase}
Steven Sivek and Raphael Zentner.
\newblock {$SU(2)$}-cyclic surgeries and the pillowcase.
\newblock arXiv:1710.01957, 2017.

\bibitem[SZ19]{sz-menagerie}
Steven Sivek and Raphael Zentner.
\newblock A menagerie of {$SU(2)$}-cyclic 3-manifolds.
\newblock arXiv:1910.13270, 2019.

\bibitem[Thu98]{thurston-fiber}
William~P. Thurston.
\newblock Hyperbolic structures on 3-manifolds, {II}: Surface groups and
  3-manifolds which fiber over the circle.
\newblock arXiv:math/9801045, 1998.

\bibitem[WZ92]{wang-zhou}
Shi~Cheng Wang and Qing Zhou.
\newblock Symmetry of knots and cyclic surgery.
\newblock {\em Trans. Amer. Math. Soc.}, 330(2):665--676, 1992.

\bibitem[Ye20]{ye-constrained}
Fan Ye.
\newblock Constrained knots in lens spaces.
\newblock arXiv:2007.04237, 2020.

\bibitem[Zen17]{zentner-simple}
Raphael Zentner.
\newblock A class of knots with simple {$SU(2)$}-representations.
\newblock {\em Selecta Math. (N.S.)}, 23(3):2219--2242, 2017.

\end{thebibliography}

\end{document}